\documentclass[10pt,a4paper]{amsart}
\usepackage[latin9]{inputenc}
\usepackage{amsthm}
\usepackage{amstext}
\usepackage{amssymb}
\usepackage{esint}
\usepackage[unicode=true,pdfusetitle,
 bookmarks=true,bookmarksnumbered=false,bookmarksopen=false,
 breaklinks=false,pdfborder={0 0 1},backref=section,colorlinks=false]
 {hyperref}

\makeatletter
\numberwithin{equation}{section}
\numberwithin{figure}{section}
\theoremstyle{plain}
\newtheorem{thm}{\protect\theoremname}
  \theoremstyle{definition}
  \newtheorem{example}[thm]{\protect\examplename}
  \theoremstyle{remark}
  \newtheorem{rem}[thm]{\protect\remarkname}
  \theoremstyle{definition}
  \newtheorem{defn}[thm]{\protect\definitionname}
  \theoremstyle{plain}
  \newtheorem{lem}[thm]{\protect\lemmaname}

\usepackage{amssymb,amsthm,amsmath,amsfonts,amscd}
\usepackage{amsaddr}
\usepackage{graphicx}
\usepackage{url}
\usepackage{color}
\usepackage[active]{srcltx}
\usepackage[matrix,arrow]{xy}
\usepackage{mathrsfs}
\usepackage{enumerate}
\usepackage{amsopn} 
\usepackage{bbm} 
\usepackage{ulem} 
\usepackage{cite}
\usepackage{hyperref}
\allowdisplaybreaks[1]
\usepackage[english]{babel}
\usepackage{bbm}
\usepackage{stmaryrd}
\usepackage{mdef}
\date{\today}

\thanks{{\bf Acknowledgements:}  We thank the DFG for support through the research project ``Random dynamical systems and regularization by noise for stochastic partial differential equations'' and through CRC 701. }

\makeatother

  \providecommand{\definitionname}{Definition}
  \providecommand{\examplename}{Example}
  \providecommand{\lemmaname}{Lemma}
  \providecommand{\remarkname}{Remark}
\providecommand{\theoremname}{Theorem}

\begin{document}

\title[SVI and regularity for degenerate SPDE]{Stochastic variational inequalities and regularity for degenerate stochastic partial differential equations}

\author{Benjamin Gess}

\address{Department of Mathematics \\
University of Chicago \\
USA }

\email{gess@uchicago.edu}

\author{Michael Röckner}

\address{Faculty of Mathematics \\
University of Bielefeld \\
Germany}

\email{roeckner@mathematik.uni-bielefeld.de}
\begin{abstract}
The regularity and characterization of solutions to degenerate, quasilinear SPDE is studied. Our results are two-fold: First, we prove regularity results for solutions to certain degenerate, quasilinear SPDE driven by Lipschitz continuous noise. In particular, this provides a characterization of solutions to such SPDE in terms of (generalized) strong solutions. Second, for the one-dimensional stochastic mean curvature flow with normal noise we adapt the notion of stochastic variational inequalities to provide a characterization of solutions previously obtained in a limiting sense only. This solves a problem left open in \cite{ESR12} and sharpens regularity properties obtained in \cite{ESRS12}.\textbf{}
\end{abstract}
\maketitle

\section{Introduction}

The study of degenerate SPDE has attracted much interest in recent years. As a model example let us consider the following stochastic $p$-Laplace type SPDE
\begin{align}
dX_{t} & =\div\phi(\nabla X_{t})dt+B(X_{t})dW_{t}\label{eq:plp-intro}\\
X_{0} & =x_{0}\nonumber 
\end{align}
with zero Dirichlet boundary conditions on bounded, smooth domains $\mcO\subseteq\R^{d}$, for some monotone $\phi\in C^{1}(\R^{d};\R^{d})$ satisfying the coercivity property
\[
\phi(\xi)\cdot\xi\ge c|\xi|^{p}\quad\forall\xi\in\R^{d}
\]
for some $c>0.$ In the following let $W$ be a trace-class Wiener process on $L^{2}(\mcO)$. In the case $p>1$ a variational approach to such SPDE (under some further assumptions) has been developed in \cite{RRW07} for initial conditions $x_{0}\in L^{2}(\mcO)$ based on the coercivity property 
\begin{equation}
_{(W^{1,p})^{*}}\<\div\phi(\nabla v),v\>_{W^{1,p}}\ge c\|v\|_{W^{1,p}}^{p}\quad\forall v\in H^{2}(\mcO),\label{eq:coerc}
\end{equation}
for some $c>0$. In the degenerate case $p=1$ these methods do not apply anymore, since the reflexivity of the energy space $W^{1,p}(\mcO)$ is lost. In particular, this difficulty appears for the stochastic mean curvature flow in one spatial dimension 
\begin{align}
dX_{t} & =\frac{\partial_{x}^{2}X_{t}}{1+(\partial_{x}X_{t})^{2}}dt+B(X_{t})dW_{t}\label{eq:MC-intro}\\
 & =\partial_{x}\arctan(\partial_{x}X_{t})+B(X_{t})dW_{t},\quad\text{on }\mcO=(0,1)\nonumber 
\end{align}
and the stochastic total variation flow
\begin{align}
dX_{t} & \in\div\left(\frac{\nabla X_{t}}{|\nabla X_{t}|}\right)dt+B(X_{t})dW_{t}.\label{eq:TV_intro-2}
\end{align}
Restricting to more regular initial data (i.e. $x_{0}\in H_{0}^{1}(\mcO)$), in \cite{ESR12} an alternative, variational approach, applicable to the stochastic mean curvature flow \eqref{eq:MC-intro} has been developed, based on the coercivity property
\begin{equation}
(\div\phi(\nabla v),v)_{H_{0}^{1}}\ge0\quad\forall v\in H^{2}(\mcO).\label{eq:coerc-1-2}
\end{equation}
This approach was subsequently generalized in \cite{GT11} to multi-valued SPDE including the stochastic total variation flow \eqref{eq:TV_intro-2}. The restriction to regular initial data ($x_{0}\in H_{0}^{1}(\mcO)$) is crucial to this approach, since it allows to work with solutions taking values in $H_{0}^{1}(\mcO)$. For general initial data $x_{0}\in L^{2}(\mcO)$ solutions to \eqref{eq:MC-intro}, \eqref{eq:TV_intro-2} could be constructed in \cite{ESR12,GT11} in a limiting sense only. That is, it has been shown that for each sequence $x_{0}^{n}\to x$ in $L^{2}(\mcO)$ with $x_{0}^{n}\in H_{0}^{1}(\mcO)$ the corresponding variational solutions $X^{n}$ converge to a limit $X$ independent of the chosen approximating sequence $x_{0}^{n}$. However, no characterization of $X$ in terms of a (generalized) solution to the corresponding SPDE could be given. In particular, this problem remained unsolved for the stochastic mean curvature flow with normal noise
\begin{align}
dX_{t} & =\partial_{x}\left(\arctan(\partial_{x}X_{t})\right)dt+\a\sqrt{1+|\partial_{x}X_{t}|^{2}}\circ d\b_{t},\label{eq:mc_normal-intro}
\end{align}
on $\mcO=(0,1)$ with periodic boundary conditions. We quote from \cite{ESR12}: {[}..{]} \textit{in view of the poor regularity of the operator $A$ {[}$A(v)=\partial_{x}\left(\arctan(\partial_{x}v)\right)${]}, a more explicit characterization of the $L^{2}([0,1])$-valued process $\hat{u}_{t}^{x}$ by some SPDE or even just an associated Kolmogorov operator on smooth finitely based test functions does not seem to be available. }For background and motivation of the stochastic mean curvature flow with normal noise we refer to \cite{ESR12,FLP14}. A numerical treatment may be found in \cite{FLP14}, higher dimensional results in \cite{SY04,LS98,LS98-2,LS00,DLN01}.

Recently, for the special case of the total variation flow with linear multiplicative noise
\begin{align}
dX_{t} & \in\div\left(\frac{\nabla X_{t}}{|\nabla X_{t}|}\right)dt+\sum_{k=1}^{\infty}f_{k}X_{t}d\b_{t}^{k},\label{eq:TV_intro-1-2}
\end{align}
with $f_{k}:\mcO\to\R$, the problem of characterizing solutions for general initial data $x_{0}\in L^{2}(\mcO)$ has been solved in \cite{BR13} by introducing the concept of stochastic variational inequalities (SVI), a notion first developed in \cite{BDPR09} for \eqref{eq:TV_intro-1-2} with additive noise. It is shown in \cite{BR13} that the limiting solutions to \eqref{eq:TV_intro-1-2} obtained in \cite{GT11} can be uniquely characterized as SVI solutions to \eqref{eq:TV_intro-1-2}. For more general SPDE of the type \eqref{eq:plp-intro}, e.g. the stochastic mean curvature flow, the problem of characterizing solutions for general initial data remained open. 

The latter problem is solved in the current paper. Our results are two-fold: First, we prove that in certain situations more regularity of solutions to degenerate SPDE than previously expected can be proved. In these cases, the concept of SVI solutions is not necessary to characterize solutions for general initial data, since we may work with (analytically) strong solutions instead (cf. Definition \ref{def:strong_soln} below). This extends regularity results for degenerate, quasilinear SPDE developed in \cite{G12} and applies to degenerate $p$-Laplace type equations
\begin{align}
dX_{t} & =\div\left(\phi(\nabla X_{t})\right)dt+\sum_{k=1}^{\infty}g_{k}(\cdot,X_{t})d\b_{t}^{k}\label{eq:TV_intro-1-1}
\end{align}
with $\phi\in C^{1}(\R^{d};\R^{d})$ satisfying appropriate conditions (cf. \eqref{eq:psi_cdt_mc} below). Among other examples, this includes the stochastic mean curvature flow with vertical noise (cf. \cite{ESR12}), i.e. \eqref{eq:TV_intro-1-1} with $\phi=\arctan$, which significantly sharpens the regularity results obtained in \cite{ESRS12}. More precisely, for general initial data $x_{0}\in L^{2}(\mcO)$ we prove 
\begin{equation}
\div\left(\phi(\nabla X_{t})\right)\in L^{2}([\tau,T]\times\O;L^{2}(\mcO))\label{eq:intro_gen_strong}
\end{equation}
for each $\tau>0$, which allows to characterize $X$ as a generalized strong solution to \eqref{eq:TV_intro-1-1} (cf. Definition \ref{def:strong_soln} below). Note that \eqref{eq:intro_gen_strong} entails a regularizing effect with respect to the initial condition, i.e. while $x_{0}\in L^{2}(\mcO)$ we observe that $X$ takes values in the domain of $\div\phi(\nabla\cdot)$, $dt\otimes d\P$-almost everywhere.

For the stochastic mean curvature flow with normal noise \eqref{eq:mc_normal-intro} additional difficulties appear, due to the irregularity of the noise. Informally rewriting \eqref{eq:mc_normal-intro} in Itô form yields
\begin{align*}
dX_{t} & =\frac{\a^{2}}{2}\partial_{x}^{2}X_{t}dt+(1-\frac{\a^{2}}{2})\frac{\partial_{x}^{2}X_{t}}{1+(\partial_{x}X_{t})^{2}}dt+\a\sqrt{1+(\partial_{x}X_{t})^{2}}d\b_{t}\\
 & =\frac{\a^{2}}{2}\partial_{x}^{2}X_{t}dt+(1-\frac{\a^{2}}{2})\partial_{x}\arctan(\partial_{x}X_{t})dt+\a\sqrt{1+(\partial_{x}X_{t})^{2}}d\b_{t}.
\end{align*}
Again, we prove new regularity results for $X$ of the type
\[
\partial_{x}\arctan(\partial_{x}X_{t})\in L^{2}([\tau,T]\times\O;L^{2}(\mcO))\quad\forall\tau>0.
\]
In contrast to \eqref{eq:TV_intro-1-1}, this improved regularity does not yield the existence of generalized strong solutions due to the additional term $\frac{\a^{2}}{2}\partial_{x}^{2}X_{t}$. We resolve this issue by introducing a notion of SVI solutions to \eqref{eq:mc_normal-intro} and by proving the existence and uniqueness of SVI solutions for each initial condition $x_{0}\in L^{2}(\mcO)$. The results thus parallel those of \cite{BR13} for the case of the stochastic total variation flow \eqref{eq:TV_intro-2}. However, in contrast to \cite{BR13} our method does not rely on a transformation into a random PDE, which leads to the restriction to linear multiplicative noise (cf. \eqref{eq:TV_intro-1-2}) in \cite{BR13}. We would also like to point out a difference in the role played by SVI solutions in the case of \eqref{eq:TV_intro-2} and \eqref{eq:mc_normal-intro}: The necessity to work with SVI solutions in \cite{BR13} is grounded in the singularity of the multi-valued sign function $\Sgn(\xi)=\frac{\xi}{|\xi|}$. More precisely, if we replace $\Sgn$ by a smooth function $\phi$ in \eqref{eq:TV_intro-2} we are in the setting of \eqref{eq:TV_intro-1-1} and generalized strong solutions exist. In contrast to this, the difficulties arising for \eqref{eq:mc_normal-intro} are due to the irregularity of the noise, rather than the irregularity of $\phi.$ Nonetheless, in both cases SVI solutions provide the means to uniquely characterize solutions for general initial data.

\subsection{Notation}

In the following let $\mcO\subseteq\R^{d}$ be a bounded set with smooth boundary. For a Hilbert space $H$ we define $C_{l.b.}^{k}(H)$ to be the set of all continuous functions on $H$ with $k$ continuous derivatives that are locally bounded on $H$. We will work with the usual Lebesgue and Sobolev spaces $L^{p}(\mcO)$, $W^{k,p}(\mcO)$ writing $L^{p}$, $H^{k,p}$ for simplicity. We further set $H^{k}(\mcO)=W^{k,2}(\mcO)$. For a function $(x,r)\mapsto g(x,r)\in C^{1}(\bar{\mcO}\times\R)$ we define the partial gradient $\nabla_{x}g(x,r):=(\partial_{x_{i}}g)_{i=1}^{d}(x,r)$ while for a function $v\in C^{1}(\mcO)$ we let $\nabla g(\cdot,v)=(\nabla_{x}g)(\cdot,v)+(\partial_{r}g)(\cdot,v)\nabla v$ be the total gradient. In the proofs, as usual, constants may change from line to line.

\section{Stochastic Parabolic quasilinear problems for linear growth functionals\label{sec:mc-pinned}}

We consider SPDE of the form
\begin{align}
dX_{t} & =\div\left(\phi(\nabla X_{t})\right)dt+\sum_{k=1}^{\infty}{\normalcolor g^{k}(\cdot,X_{t})}d\b_{t}^{k},\label{eq:MC-pinned}\\
X_{0} & =x_{0}\nonumber 
\end{align}
with zero Dirichlet boundary conditions on bounded, convex, smooth domains $\mcO\subseteq\R^{d}$ with $d\le6$. Here, $\b^{k}$ are independent Brownian motions and
\begin{enumerate}
\item [(B)] $g^{k}\in C^{1}(\bar{\mcO}\times\R)$ with $g^{k}(x,0)=0$ for all $x\in\partial\mcO$ and
\[
\sum_{k=1}^{\infty}\mu_{k}<\infty,
\]
where 
\[
\mu_{k}:=\left\Vert \partial_{r}g^{k}\right\Vert _{C^{0}(\bar{\mcO}\times\R)}^{2}+\left\Vert \frac{|\nabla_{x}g^{k}(x,r)|}{1+|r|}\right\Vert _{C^{0}(\bar{\mcO}\times\R)}^{2}.
\]

\end{enumerate}
In particular, this includes additive noise, i.e. $g^{k}\in C_{0}^{1}(\bar{\mcO})$ and linear multiplicative noise, i.e. $g^{k}(x,r)=\vp^{k}(x)r$ for $\vp^{k}\in C^{1}(\bar{\mcO})$. We assume $\phi=\nabla\psi$ for some non-negative, convex, radial $\psi\in C^{2}(\R^{d};\R)$ with $\psi(0)=0$, $\phi,D\phi$ being Lipschitz and 
\begin{equation}
\begin{array}{ccccc}
c|\xi|-C & \le & \psi(\xi) & \le & C(1+|\xi|)\\
c\psi(\xi)-C & \le & \phi(\xi)\cdot\xi\\
 &  & |D\phi(\xi)||\xi|^{2} & \le & C(1+\psi(\xi)),
\end{array}\label{eq:psi_cdt_mc}
\end{equation}
for all $\xi\in\R^{d}$ and some constants $c>0,$ $C\ge0$. Then the recession function $\psi^{0}$ defined by 
\[
\psi^{0}(\xi):=\lim_{t\downarrow0}\psi\left(\frac{\xi}{t}\right)t
\]
exists and is finite since $\psi$ is of linear growth and $t\mapsto\psi\left(\frac{\xi}{t}\right)t$ is non-increasing. 
\begin{example}
Typical examples of $\psi$ are 
\begin{enumerate}
\item Mean curvature flow in one dimension (cf. \cite{ESR12,ESRS12}):
\begin{align*}
\psi(\xi) & =\xi\arctan(\xi)-\frac{1}{2}\log(\xi^{2}+1)\\
\dot{\psi}(\xi) & =\phi(\xi):=\arctan(\xi)\\
\dot{\phi}(\xi) & =\frac{1}{1+\xi^{2}},\quad\xi\in\R.
\end{align*}

\item Minimal surface/image denoising (cf. \cite{GT11,BR13,GR92,KOJ05}):
\begin{align*}
\psi(\xi) & =\sqrt{\ve+|\xi|^{2}}\\
\nabla\psi(\xi) & =\phi(\xi)=\frac{\xi}{\sqrt{\ve+|\xi|^{2}}}\\
D\phi(\xi) & =\frac{1}{\sqrt{\ve+|\xi|^{2}}}\left(-\frac{\xi\otimes\xi}{\ve+|\xi|^{2}}+Id\right),\quad\xi\in\R^{d},\ve>0,d\ge1.
\end{align*}
We note that $\psi(\xi)=\sqrt{\ve+|\xi|^{2}}$ may be considered as a smooth approximation to the total variation functional $\psi(\xi)=|\xi|$. In \cite{GR92} a general class of nonlinearities $\psi$ has been considered with regard to application in image restoration (cf. also \cite{ROF92}). 
\end{enumerate}
\end{example}
\begin{rem}
The same methods as developed in this section may be applied to nonlinearities arising in generalized Newtonian fluids:
\begin{align*}
\psi(\xi) & =\frac{1}{p}(1+|\xi|^{2})^{\frac{p}{2}}\\
\nabla\psi(\xi) & =\phi(\xi)=(1+|\xi|^{2})^{\frac{p-2}{2}}\xi\\
D\phi(\xi) & =(1+|\xi|^{2})^{\frac{p-2}{2}}\left((p-2)\frac{\xi\otimes\xi}{1+|\xi|^{2}}+Id\right),\quad\xi\in\R^{d},
\end{align*}
with $p\in(1,2)$. For simplicity we restrict to linear growth functionals satisfying \eqref{eq:psi_cdt_mc}.
\end{rem}
In this section we will work with the Hilbert spaces $H=L^{2}(\mcO)$ and $V=H_{0}^{1}(\mcO)$. For $v\in H$ we set $B(v)(h):=\sum_{k=1}^{\infty}g^{k}(\cdot,v)(e_{k},h)_{2}$, where $e_{i}\in H$ is an orthonormal basis of $H$. Then $B:H\to L_{2}(H,H)$ is of linear growth, i.e.
\begin{align*}
\|B(v)\|_{L_{2}(H,H)}^{2} & =\sum_{k=1}^{\infty}\|g^{k}(\cdot,v)\|_{H}^{2}\\
 & \le C(1+\|v\|_{H}^{2})\sum_{k=1}^{\infty}\mu_{k},
\end{align*}
for all $v\in H$. Similarly one shows that $H\ni v\mapsto B(v)\in L_{2}(H,H)$ is Lipschitz. Following \cite{A85} we define
\[
\vp(v):=\begin{cases}
\int_{\mcO}\psi(Dv)dx+\int_{\partial\mcO}\psi^{0}(\nu(x)v(x))H^{d-1}(dx) & \text{if }v\in L^{2}\cap BV\\
+\infty & \text{if }v\in L^{2}\setminus BV,
\end{cases}
\]
where $\int_{\mcO}\psi(\mu)dx$ for a bounded Radon measure $\mu$ with Lebesgue decomposition $\mu=\mu^{a}+\mu^{s}$ is defined as in \cite{A85}, i.e. 
\[
\int_{\mcO}\psi(\mu)dx=\int_{\mcO}\psi(\mu^{a})dx+\int_{\mcO}\psi^{0}\left(\frac{d\mu}{d|\mu|}\right)d|\mu|^{s}\text{ }
\]
and $\nu$ is the outer normal on $\partial\mcO.$ For $u\in BV$ we consider the Lebesgue decomposition $Du=D^{a}u+D^{s}u$ where $D^{a}u$ denotes the absolutely continuous part of $Du$ with respect to the Lebesgue measure with density $\nabla u$. Obviously, $\vp$ is convex and $\vp$ restricted to $W_{0}^{1,1}$ is continuous. Furthermore it follows from \cite[Fact 3.5]{A83} that $\vp$ is the lower-semicontinuous hull on $L^{2}$ of $\vp_{|W_{0}^{1,1}\cap L^{2}}$. In the sequel $\partial\vp:=\partial_{H}\vp$ denotes the subgradient of $\vp$ on $H$.
\begin{rem}
Under certain additional assumptions on $\psi$ (cf. \cite{ACM02} for details) the following characterization of the subgradient of $\vp$ has been given in \cite{ACM02}: We have $(u,v)\in\partial\vp$ iff $u\in L^{2}\cap BV$, $v\in L^{2}$, $\phi(\nabla u)\in X(\mcO)=\{z\in L^{\infty}(\mcO;\R^{d}):\,\div(z)\in L^{1}(\mcO)\}$ and
\begin{align*}
v & =-\div\phi(\nabla u),\\
\phi(\nabla u)\cdot D^{s}u & =\psi^{0}(D^{s}u),\\
{}[\phi(\nabla u),\nu](x) & \in-\Sgn(u(x))\psi^{0}(\nu(x)),\quad H^{d-1}-\text{a.e.},
\end{align*}
where $[z,\nu]$ denotes the weak trace for $z\in X(\mcO)$. 
\end{rem}
Our arguments will, however, not rely on this identification of the subgradient. We note that $\vp_{|H_{0}^{1}}$ is Gateaux-differentiable with derivative 
\[
D\vp_{|H_{0}^{1}}(u)(v)=\int_{\mcO}\phi(\nabla u)\nabla vdx.
\]
Using that $\vp$ is the lower-semicontinuous hull of $\vp_{|H_{0}^{1}}$ this implies: If $u\in H_{0}^{1}$ then 
\begin{equation}
\partial_{H_{0}^{1}}\vp(u)=\div\phi(\nabla u)\in H^{-1},\label{eq:H1-subgradient}
\end{equation}
where $\partial_{H_{0}^{1}}\vp:H_{0}^{1}\to H^{-1}$ denotes the subgradient of $\vp_{|H_{0}^{1}}$. If, in addition, $\div\phi(\nabla u)\in L^{2}$ then 
\[
\partial\vp(u)=-\div\phi(\nabla u).
\]
Hence, we may rewrite \eqref{eq:MC-pinned} in the \textit{relaxed} form
\begin{align}
dX_{t} & \in-\partial\vp(X_{t})dt+B(X_{t})dW_{t},\label{eq:MC-pinned-2}\\
X_{0} & =x_{0}.\nonumber 
\end{align}
In \cite{ESR12,GT11} a variational framework for regular initial conditions $x_{0}\in H_{0}^{1}$ has been developed, while for general initial data $x_{0}\in H$ solutions to \eqref{eq:MC-pinned} could only be constructed in a limiting sense. In the following we will introduce stronger notions of solutions to \eqref{eq:MC-pinned} based on the subgradient formulation \eqref{eq:MC-pinned-2}. The main result will be the proof of existence of solutions in this stronger sense. This includes the proof of higher regularity of solutions.
\begin{defn}
\label{def:strong_soln}Let $x_{0}\in L^{2}(\Omega;H).$ An $H$-continuous, $\mcF_{t}$-adapted process $X\in L^{2}(\Omega;C([0,T];H))$ for which there exists a selection $\eta\in-\partial\vp(X)$, $dt\otimes d\P$-a.e. is said to be a 
\begin{enumerate}
\item strong solution to \eqref{eq:MC-pinned} if 
\[
\eta\in L^{2}([0,T]\times\Omega;H)
\]
and $\P$-a.s.
\[
X_{t}=x_{0}+\int_{0}^{t}\eta_{r}dr+\int_{0}^{t}B(X_{r})dW_{r},\quad\forall t\in[0,T],
\]
as an equation in $H$.
\item generalized strong solution to \eqref{eq:MC-pinned} if 
\[
\eta\in L^{2}([\tau,T]\times\Omega;H),\quad\forall\tau>0
\]
and $\P$-a.s.
\[
X_{t}=X_{\tau}+\int_{\tau}^{t}\eta_{r}dr+\int_{\tau}^{t}B(X_{r})dW_{r},\quad\forall t\in[\tau,T],
\]
for all $\tau>0$, as an equation in $H$.
\end{enumerate}
\end{defn}

We prove the existence of strong solutions to \eqref{eq:MC-pinned} for initial conditions $x_{0}\in L^{2}(\Omega;H)$ satisfying $\E\vp(x_{0})<\infty$. Moreover, we will prove regularizing properties with respect to the initial condition due to the subgradient structure of the drift. This allows to characterize solutions for initial conditions $x_{0}\in L^{2}(\O;H)$, which before were constructed in a limiting sense only.
\begin{thm}
\label{thm:mc-pinned}Let $x_{0}\in L^{2}(\Omega;H)$. 
\begin{enumerate}
\item There is a unique generalized strong solution $(X,\eta)$ to \eqref{eq:MC-pinned} and $(X,\eta)$ satisfies
\begin{align*}
 & \E t\vp(X_{t})+\E\int_{0}^{t}r\|\eta_{r}\|_{H}^{2}dr\le C\left(\E\|x_{0}\|_{H}^{2}+1\right)\quad\forall t\in[0,T].
\end{align*}

\item If \textup{$\E\vp(x_{0})<\infty$.} Then, there is a unique strong solution $(X,\eta)$ to \eqref{eq:MC-pinned} satisfying
\begin{equation}
\E\vp(X_{t})+\E\int_{0}^{t}\|\eta_{r}\|_{H}^{2}dr\le\E\vp(x_{0})+C\quad\forall t\in[0,T].\label{approx_strong_S-3}
\end{equation}

\end{enumerate}
The (generalized) strong solution $ $$X$ coincides with the limit solution constructed in \cite{ESR12}.
\end{thm}
The proof of Theorem \ref{thm:mc-pinned} proceeds in several steps. In order to justify our calculations we will consider a vanishing viscosity approximation to \eqref{eq:MC-pinned}. Let us first assume $x_{0}\in L^{2}(\O;H_{0}^{1})$. We will remove this restriction in the end of the proof. We consider the following non-degenerate approximation:
\begin{align}
dX_{t}^{\ve} & \in\ve\D X_{t}^{\ve}dt+\div\phi(\nabla X_{t}^{\ve})dt+B(X_{t}^{\ve})dW_{t},\label{eq:vis_approx}\\
X_{0}^{\ve} & =x_{0}.\nonumber 
\end{align}
 For $v\in H$ we define
\[
\vp^{\ve}(v):=\begin{cases}
\frac{\ve}{2}\int_{\mcO}|\nabla v|^{2}dx+\int_{\mcO}\psi(\nabla v)dx, & \text{for }v\in H_{0}^{1}\\
+\infty, & \text{otherwise.}
\end{cases}
\]
Note that $\vp^{\ve}\in C^{2}(H_{0}^{1}\cap H^{2})$ with Lipschitz continuous derivatives given by 
\begin{align*}
D\vp^{\ve}(v)(h) & =\ve\int_{\mcO}\nabla v\cdot\nabla hdx+\int_{\mcO}\phi(\nabla v)\cdot\nabla hdx,\\
D^{2}\vp^{\ve}(v)(g,h) & =\ve\int_{\mcO}\nabla h\cdot\nabla gdx+\int_{\mcO}\nabla h\cdot D\phi(\nabla v)\nabla gdx.
\end{align*}
 To check the claimed continuity we note that
\begin{align*}
D\vp^{\ve}(v)(h)-D\vp^{\ve}(w)(h) & =\ve\int_{\mcO}(\nabla v-\nabla w)\cdot\nabla hdx+\int_{\mcO}(\phi(\nabla v)-\phi(\nabla w))\cdot\nabla hdx\\
 & \le\ve\|v-w\|_{H_{0}^{1}}\|h\|_{H_{0}^{1}}+\|\phi(\nabla v)-\phi(\nabla w)\|_{2}\|h\|_{H_{0}^{1}}\\
 & \lesssim(\ve+1)\|v-w\|_{H_{0}^{1}}\|h\|_{H_{0}^{1}}
\end{align*}
and
\begin{align}
D^{2}\vp^{\ve}(v)(g,h)-D^{2}\vp^{\ve}(w)(g,h) & =\int_{\mcO}\nabla h\cdot\left(D\phi(\nabla v)-D\phi(\nabla w)\right)\nabla gdx\nonumber \\
 & \le\|\nabla h\|_{3}\|\nabla g\|_{3}\|D\phi(\nabla v)-D\phi(\nabla w)\|_{3}\label{eq:lip_derivatives}\\
 & \lesssim\|\nabla h\|_{3}\|\nabla g\|_{3}\|\nabla v-\nabla w\|_{3}\nonumber \\
 & \lesssim\|h\|_{H_{0}^{1}\cap H^{2}}\|g\|_{H_{0}^{1}\cap H^{2}}\|v-w\|_{H_{0}^{1}\cap H^{2}},\nonumber 
\end{align}
where we use the Sobolev embedding $H^{1}\hookrightarrow L^{3}$ due to $d\le6$. Hence, $\vp^{\ve}\in C^{2}(H_{0}^{1}\cap H^{2})$ with Lipschitz continuous derivatives. Moreover, $\vp^{\ve}$ is a convex, lower-semicontinuous  function on $H_{0}^{1}$ with (single-valued) subgradient given by
\[
A^{\ve}(v):=-\partial_{H_{0}^{1}}\vp^{\ve}(v)=\ve\D v+\div\phi(\nabla v)\in H^{-1},\quad\text{for }v\in H_{0}^{1}.
\]

By \cite{PR07} we know that there is a unique, variational solution $X^{\ve}\in L^{2}(\Omega;C([0,T];H))\cap L^{2}([0,T]\times\Omega;H_{0}^{1})$ to \eqref{eq:vis_approx} satisfying the estimate
\begin{align*}
\E\sup_{t\in[0,T]}\|X_{t}^{\ve}\|_{H}^{2} & \le C\E\|x_{0}\|_{H}^{2}.
\end{align*}
We will now prove that in fact $X^{\ve}$ is a strong solution in the following sense:
\begin{lem}
\label{lem:H1-bound}For each $\ve>0$ we have $X^{\ve}\in L^{2}([0,T]\times\Omega;H^{2}\cap H_{0}^{1})$ and
\[
\E\sup_{t\in[0,T]}e^{-Ct}\|X_{t}^{\ve}\|_{H_{0}^{1}}^{2}+4\ve\int_{\text{0}}^{t}\E e^{-Cr}\|X_{r}^{\ve}\|_{H^{2}}^{2}dr\le C\left(\E\|x_{0}\|_{H_{0}^{1}}^{2}+1\right),
\]
for some constant $C$ independent of $\ve>0.$\end{lem}
\begin{proof}
In the following we let $(e_{i})_{i=1}^{\infty}$ be an orthonormal basis of eigenvectors of $-\D$ on $L^{2}$. We further let $P^{n}:H\to\text{span}\{e_{1},\dots,e_{n}\}$ be the orthogonal projection onto the span of the first $n$ eigenvectors. We recall that the unique variational solution $X^{\ve}$ to \eqref{eq:vis_approx} is constructed in \cite{PR07} as a (weak) limit in $L^{2}([0,T]\times\O;H_{0}^{1})$ of the solutions to the following Galerkin approximation
\begin{align*}
dX_{t}^{n} & =\ve P^{n}\D X_{t}^{n}dt+P^{n}\div\phi(\nabla X_{t}^{n})dt+P^{n}B(X_{t}^{n})dW_{t}^{n},\\
X_{0}^{n} & =P^{n}x_{0}.
\end{align*}
Itô's formula then yields
\begin{align*}
\|X_{t}^{n}\|_{H_{0}^{1}}^{2} & =\|P^{n}x_{0}\|_{H_{0}^{1}}^{2}+2\int_{0}^{t}(X_{r}^{n},\ve P^{n}\D X_{r}^{n}+P^{n}\div\phi(\nabla X_{t}^{n}))_{H_{0}^{1}}dr\\
 & +2\int_{0}^{t}(X_{r}^{n},P^{n}B(X_{r}^{n})dW_{r}^{n})_{H_{0}^{1}}dr+\int_{0}^{t}\|P^{n}B(X_{r}^{n})\|_{L_{2}(H,H_{0}^{1})}^{2}dr\\
 & =\|P^{n}x_{0}\|_{H_{0}^{1}}^{2}-2\ve\int_{0}^{t}\|\D X_{r}^{n}\|_{2}^{2}dr+2\int_{0}^{t}(X_{r}^{n},P^{n}\div\phi(\nabla X_{t}^{n}))_{H_{0}^{1}}dr\\
 & +2\int_{0}^{t}(X_{r}^{n},P^{n}B(X_{r}^{n})dW_{r}^{n})_{H_{0}^{1}}dr+\int_{0}^{t}\|P^{n}B(X_{r}^{n})\|_{L_{2}(H,H_{0}^{1})}^{2}dr.
\end{align*}
For $v\in H_{0}^{1}$ smooth we note
\begin{align*}
(v,\div\phi(\nabla v))_{H_{0}^{1}} & =(-\D v,\div\phi(\nabla v))_{2}\\
 & =\lim_{n\to\infty}(T^{n}v,\div\phi(\nabla v))_{2}\\
 & =-\lim_{n\to\infty}n(J^{\frac{1}{n}}v-v,\div\phi(\nabla v))_{2},
\end{align*}
where $J^{\frac{1}{n}}:=(1-\frac{1}{n}\D)^{-1}$ is the resolvent and $T^{n}=n(1-J^{\frac{1}{n}})$ is the Yosida-approximation of $-\D$ on $L^{2}$. Since $\div\phi(\nabla v)=-\partial\vp(v)$ we obtain
\begin{align*}
(v,\div\phi(\nabla v))_{H_{0}^{1}} & \le\lim_{n\to\infty}n(\vp(J^{\frac{1}{n}}v)-\vp(v)).
\end{align*}
We note that
\begin{align*}
\vp(J^{\frac{1}{n}}v) & =\int_{\mcO}\psi(\nabla J^{\frac{1}{n}}v)dx\le\int_{\mcO}\psi(\nabla v)dx=\vp(v)
\end{align*}
due to \cite[Proposition 8.2]{BR13} (using that $\psi$ is radial) and thus (choosing $v=X_{r}^{n}$)
\begin{align*}
(X_{r}^{n},\div\phi(\nabla X_{r}^{n}))_{H_{0}^{1}} & \le0.
\end{align*}
Using this, 
\begin{align*}
\|B(v)\|_{L_{2}(H,H_{0}^{1})}^{2} & =\sum_{k=1}^{\infty}\|g^{k}(x,v)\|_{H_{0}^{1}}^{2}\\
 & =\sum_{k=1}^{\infty}\|\nabla_{x}g^{k}(x,v)+\partial_{r}g^{k}(x,v)\nabla v\|_{L^{2}}^{2}\\
 & \le C(1+\|v\|_{H_{0}^{1}}^{2})\sum_{k=1}^{\infty}\mu_{k}\quad\forall v\in H_{0}^{1}
\end{align*}
and the Burkholder-Davis-Gundy inequality yields
\begin{align}
\frac{1}{2}\E\sup_{t\in[0,T]}e^{-Ct}\|X_{t}^{n}\|_{H_{0}^{1}}^{2} & \le\E\|x_{0}\|_{H_{0}^{1}}^{2}-2\ve\E\int_{0}^{T}e^{-Cr}\|\D X_{r}^{n}\|_{2}^{2}dr+C,\label{eq:galerkin_ineq}
\end{align}
for some $C>0$ large enough. Hence, $X^{n}$ is uniformly bounded in $L^{2}([0,T]\times\O;H^{2})$ and $L^{2}(\O;L^{\infty}([0,T];H_{0}^{1}))$ and we may extract a weakly (weak$^{*}$ resp.) convergent subsequence (for simplicity we stick with the notation $X^{n}$). Therefore, we have
\begin{align*}
X^{n} & \rightharpoonup X,\quad\text{in }L^{2}([0,T]\times\O;H^{2}),\\
X^{n} & \rightharpoonup^{*}X,\quad\text{in }L^{2}(\O;L^{\infty}([0,T];H_{0}^{1})),
\end{align*}
for $n\to\infty$. By weak lower semicontinuity of the norms we may pass to the limit in \eqref{eq:galerkin_ineq} which yields the claim.\end{proof}
\begin{lem}
\label{lem:H-bound}For each $\ve>0$ we have $\vp^{\ve}(X^{\ve})\in L^{1}([0,T]\times\Omega)$ with
\[
\sup_{t\in[0,T]}\E\|X_{t}^{\ve}\|_{H}^{2}+\E\int_{0}^{T}e^{-Cr}\vp^{\ve}(X_{r}^{\ve})dr\le C\left(\E\|x_{0}\|_{H}^{2}+1\right),
\]
for some constant $C$ independent of $\ve>0.$ \end{lem}
\begin{proof}
Note that, using \eqref{eq:psi_cdt_mc}
\begin{align}
\ _{V^{*}}\<A^{\ve}(v),v\>_{V} & =-\int_{\mcO}\left(\ve|\nabla v|^{2}+\phi(\nabla v)\cdot\nabla v\right)dx\nonumber \\
 & \le-\int_{\mcO}\left(\ve|\nabla v|^{2}+c\psi(\nabla v)+C\right)dx\label{eq:coerc_of_Aeps}\\
 & \le-c\vp^{\ve}(v)+C,\nonumber 
\end{align}
for all $v\in V$. By Itô's formula we have
\begin{align*}
 & \E e^{-Kt}\|X_{t}^{\ve}\|_{H}^{2}\\
= & \E\|x_{0}^{\ve}\|_{H}^{2}+2\E\int_{0}^{t}e^{-Kr}\ _{V^{*}}\<A^{\ve}(X_{r}^{\ve}),X_{r}^{\ve}\>_{V}+e^{-Kr}\|B(X_{r}^{\ve})\|_{L_{2}(H,H)}^{2}dr\\
 & -K\int_{0}^{t}e^{-Kr}\|X_{r}^{\ve}\|_{H}^{2}dr\\
\le & \E\|x_{0}^{\ve}\|_{H}^{2}-2\E\int_{0}^{t}ce^{-Kr}\vp^{\ve}(X_{r}^{\ve})+Ce^{-Kr}\|X_{r}^{\ve}\|_{H}^{2}dr-K\int_{0}^{t}e^{-Kr}\|X_{r}^{\ve}\|_{H}^{2}dr+C.
\end{align*}
Choosing $K$ large enough yields the claim.
\end{proof}

Based on the strong solution property of $X^{\ve}$ we derive the key estimate in the following
\begin{lem}
\label{lem:strong_approx}Let $x_{0}\in L^{2}(\O;H_{0}^{1})$. For all $\ve>0$ we have
\begin{align}
 & \E t\vp^{\ve}(X_{t}^{\ve})+\E\int_{0}^{t}r\|\ve\D X_{r}^{\ve}+\div\phi(\nabla X_{r}^{\ve})\|_{H}^{2}dr\label{eq:approx_strong_H}\\
 & \le\ve C\left(\E\|x_{0}\|_{H_{0}^{1}}^{2}+1\right)+C\left(\E\|x_{0}\|_{H}^{2}+1\right).\nonumber 
\end{align}
and
\begin{equation}
\E\vp^{\ve}(X_{t}^{\ve})+\E\int_{0}^{t}\|\ve\D X_{r}^{\ve}+\div\phi(\nabla X_{r}^{\ve})\|_{H}^{2}dr\le\E\vp^{\ve}(x_{0})+C,\label{approx_strong_S}
\end{equation}
for some constant $C>0$.\end{lem}
\begin{proof}
We first prove \eqref{eq:approx_strong_H}: Let $J^{\l}$ be the resolvent of $-\D$ on $L^{2}$, i.e. $J^{\l}:=(1-\l\D)^{-1}$. We define $\vp^{\ve,\l}:=\vp^{\ve}\circ J^{\l}.$ Since $J^{\l}:H\to H^{2}\cap H_{0}^{1}$ is a linear, continuous operator we have $\vp^{\ve,\l}\in C^{2}(H)$ with Lipschitz continuous derivatives (cf. \eqref{eq:lip_derivatives}) given by
\begin{align*}
D\vp^{\ve,\l}(v)(h) & =D\vp^{\ve}(J^{\l}v)(J^{\l}h)\\
 & =\int_{\mcO}\left(\ve(\nabla J^{\l}v)\cdot(\nabla J^{\l}h)+\phi(\nabla J^{\l}v)\cdot\nabla J^{\l}h\right)dx\\
 & =-\left(\ve\D J^{\l}v+\div\phi(\nabla J^{\l}v),J^{\l}h\right)_{H},\\
D^{2}\vp^{\ve,\l}(v)(g,h) & =D^{2}\vp^{\ve}(J^{\l}v)(J^{\l}g,J^{\l}h)\\
 & =\int_{\mcO}\ve(\nabla J^{\l}h)\cdot(\nabla J^{\l}g)dx+\int_{\mcO}(\nabla J^{\l}h)\cdot D\phi(\nabla J^{\l}v)(\nabla J^{\l}g)dx.
\end{align*}
For \eqref{eq:approx_strong_H}: We apply Itô's formula to $t\vp^{\ve,\l}(X_{t}^{\ve})$ to get:
\begin{align*}
 & \E t\vp^{\ve,\l}(X_{t}^{\ve})\\
 & =\E\int_{0}^{t}r(D\vp^{\ve,\l}(X_{r}^{\ve}),\ve\D X_{r}^{\ve}+\div\phi(\nabla X_{r}^{\ve}))_{H}dr\\
 & +\frac{1}{2}\E\int_{0}^{t}rTr[D^{2}\vp^{\ve,\l}(X_{r}^{\ve})B(X_{r}^{\ve})B^{*}(X_{r}^{\ve})]dr+\E\int_{0}^{t}\vp^{\ve,\l}(X_{r}^{\ve})dr\\
 & =-\E\int_{0}^{t}r(\ve\D J^{\l}(X_{r}^{\ve})+\div\phi(\nabla J^{\l}X_{r}^{\ve}),\ve J^{\l}\D X_{r}^{\ve}+J^{\l}\div\phi(\nabla X_{r}^{\ve}))_{H}dr\\
 & +\frac{\ve}{2}\sum_{k=1}^{\infty}\E\int_{0}^{t}r\int_{\mcO}|\nabla J^{\l}g^{k}(x,X_{r}^{\ve}(x))|^{2}dxdr\\
 & +\frac{1}{2}\sum_{k=1}^{\infty}\E\int_{0}^{t}r\int_{\mcO}(\nabla J^{\l}g^{k}(x,X_{r}^{\ve}(x)))\cdot D\phi(\nabla J^{\l}X_{r}^{\ve}(x))(\nabla J^{\l}g^{k}(x,X_{r}^{\ve}(x)))dxdr\\
 & +\E\int_{0}^{t}\vp^{\ve,\l}(X_{r}^{\ve})dr.
\end{align*}
We first note that
\begin{align*}
\int_{\mcO}|\nabla J^{\l}g^{k}(x,X_{r}^{\ve}(x))|^{2}dx & \le\int_{\mcO}|\nabla g^{k}(x,X_{r}^{\ve}(x))|^{2}dx\\
 & \le C\int_{\mcO}|\nabla_{x}g^{k}(x,X_{r}^{\ve}(x))+\partial_{r}g^{k}(x,X_{r}^{\ve}(x))\nabla X_{r}^{\ve}(x)|^{2}dx\\
 & \le C\mu_{k}(1+\|X_{r}^{\ve}\|_{H_{0}^{1}}^{2}).
\end{align*}
Moreover,
\begin{align*}
|(\nabla J^{\l}g^{k}(\cdot,v))\cdot D\phi(\nabla J^{\l}v)(\nabla J^{\l}g^{k}(\cdot,v))| & \le|\nabla J^{\l}g^{k}(\cdot,v)|^{2}|D\phi|(\nabla J^{\l}v).
\end{align*}
We note $J^{\l}v\to v$ in $H_{0}^{1}\cap H^{2}$ for $v\in H_{0}^{1}\cap H^{2}$ and thus $\nabla J^{\l}v\to\nabla v$ in $ $$H^{1}$ for $\l\to0$. Since $D\phi$ is Lipschitz we have $|D\phi(\nabla J^{\l}v)-D\phi(\nabla v)|\to0$ in $L^{2}([0,T]\times\O;H^{1})$ for all $v\in L^{2}([0,T]\times\O;H_{0}^{1}\cap H^{2})$ for $\l\to0$. Moreover, $|\nabla J^{\l}g^{k}(\cdot,v)-\nabla g^{k}(\cdot,v)|\to0$ in $L^{2}([0,T]\times\O;H)$ for $\l\to0$. Hence, (using $H^{1}\hookrightarrow L^{3}$, $D\phi\in C_{b}^{0}$ and \eqref{eq:psi_cdt_mc})
\begin{align*}
 & \lim_{\l}\E\int_{0}^{t}r\int_{\mcO}(\nabla J^{\l}g^{k}(\cdot,v_{r}))\cdot D\phi(\nabla J^{\l}v_{r})(\nabla J^{\l}g^{k}(\cdot,v_{r}))dxdr\\
 & \le\E\int_{0}^{t}r\int_{\mcO}|\nabla g^{k}(\cdot,v_{r}))|^{2}|D\phi|(\nabla v_{r})dxdr\\
 & \le C\E\int_{0}^{t}r\int_{\mcO}\left(|\nabla_{x}g^{k}(\cdot,v_{r})|^{2}+|\partial_{r}g^{k}(\cdot,v_{r})\nabla v_{r}|^{2}\right)|D\phi|(\nabla v_{r})dxdr\\
 & \le C\mu_{k}\E\int_{0}^{t}r\int_{\mcO}(1+|v_{r}|^{2}+|\nabla v_{r}|^{2})|D\phi|(\nabla v_{r})dxdr\\
 & \le C\mu_{k}\E\int_{0}^{t}r\int_{\mcO}(1+|v_{r}|^{2}+\psi(\nabla v_{r}))dxdr\\
 & \le C\mu_{k}\left(1+\E\int_{0}^{t}r\|v_{r}\|_{H}^{2}dr+\E\int_{0}^{t}r\vp^{\ve}(v_{r})dr\right),
\end{align*}
for all $v\in L^{2}([0,T]\times\O;H_{0}^{1}\cap H^{2})$. Moreover,
\begin{align*}
 & \left|\E\int_{0}^{t}r\int_{\mcO}(\nabla J^{\l}g^{k}(\cdot,v_{r}))\cdot D\phi(\nabla J^{\l}(v_{r}))(\nabla J^{\l}g^{k}(\cdot,v))dxdr\right|\\
 & \le C\mu_{k}\E\int_{0}^{t}r(1+\|v_{r}\|_{H_{0}^{1}}^{2})dr
\end{align*}
and the right hand side is summable in $k$. Hence, dominated convergence applies and we obtain 
\begin{align*}
 & \lim_{\l\to0}\frac{\ve}{2}\sum_{k=1}^{\infty}\E\int_{0}^{t}r\int_{\mcO}|\nabla J^{\l}g^{k}(x,X_{r}^{\ve}(x))|^{2}dxdr\\
 & +\lim_{\l\to0}\frac{1}{2}\sum_{k=1}^{\infty}\E\int_{0}^{t}r\int_{\mcO}(\nabla J^{\l}g^{k}(x,X_{r}^{\ve}(x)))\cdot D\phi(\nabla J^{\l}(X_{r}^{\ve}(x)))(\nabla J^{\l}g^{k}(x,X_{r}^{\ve}(x)))dxdr\\
 & \le C\left(1+\frac{\ve}{2}\E\int_{0}^{t}r(1+\|X_{r}^{\ve}\|_{H_{0}^{1}}^{2})dr+\E\int_{0}^{t}r\|X_{r}^{\ve}\|_{H}^{2}dr+\E\int_{0}^{t}r\vp^{\ve}(X_{r}^{\ve})dr\right).
\end{align*}
Since $X^{\ve}\in L^{2}([0,T]\times\Omega;H^{2}\cap H_{0}^{1})$ we have (using dominated convergence) $\nabla J^{\l}X^{\ve}\to\nabla X^{\ve}$ in $L^{2}([0,T]\times\Omega;H^{1})$ for $\l\to0$. Since $\phi,D\phi$ are Lipschitz this implies $\phi(\nabla J^{\l}X^{\ve})\to\phi(\nabla X^{\ve})$ for $\l\to0$ in $L^{2}([0,T]\times\Omega;H^{1})$. Moreover, $\D J^{\l}X^{\ve}=J^{\l}\D X^{\ve}\to\D X^{\ve}$ and $J^{\l}\div\phi(\nabla X_{r}^{\ve})\to\div\phi(\nabla X_{r}^{\ve})$ in $L^{2}([0,T]\times\Omega;H)$. Hence, we obtain
\begin{align*}
 & \lim_{\l\to0}-\E\int_{0}^{t}r(\ve\D J^{\l}(X_{r}^{\ve})+\div\phi(\nabla J^{\l}X_{r}^{\ve}),\ve J^{\l}\D X_{r}^{\ve}+J^{\l}\div\phi(\nabla X_{r}^{\ve}))_{H}dr\\
 & =-\E\int_{0}^{t}r\|\ve\D X_{r}^{\ve}+\div\phi(\nabla X_{r}^{\ve})\|_{H}^{2}dr.
\end{align*}
Since $\vp^{\ve,\l}(X_{r}^{\ve})=\vp^{\ve}(J_{\l}X_{r}^{\ve})$, $J_{\l}X^{\ve}\to X^{\ve}$ in $L^{2}([0,T]\times\Omega;H^{2}\cap H_{0}^{1})$, $ $ $J_{\l}X_{t}^{\ve}\to X_{t}^{\ve}$ in $L^{2}(\Omega;H_{0}^{1})$ for all $t\in[0,T]$ and since $\vp^{\ve}$ is continuous on $H_{0}^{1}$ and 
\[
|\vp^{\ve,\l}(X_{t}^{\ve})|\le C(1+\|X_{t}^{\ve}\|_{H_{0}^{1}}^{2}),
\]
by Lebesgue's dominated convergence theorem we get
\begin{align*}
 & \lim_{\l\to0}\E t\vp^{\ve,\l}(X_{t}^{\ve})=\E t\vp^{\ve}(X_{t}^{\ve})
\end{align*}
and 
\[
\lim_{\l\to0}\E\int_{0}^{t}\vp^{\ve,\l}(X_{r}^{\ve})dr=\E\int_{0}^{t}\vp^{\ve}(X_{r}^{\ve})dr.
\]
Putting these estimates together yields
\begin{align*}
\E t\vp^{\ve}(X_{t}^{\ve}) & \le-\E\int_{0}^{t}r\|\ve\D X_{r}^{\ve}+\div\phi(\nabla X_{r}^{\ve})\|_{H}^{2}dr\\
 & +C\left(1+\ve\E\int_{0}^{t}r(1+\|X_{r}^{\ve}\|_{H_{0}^{1}}^{2})dr+\E\int_{0}^{t}r\|X_{r}^{\ve}\|_{H}^{2}dr+\E\int_{0}^{t}r\vp^{\ve}(X_{r}^{\ve})dr\right)\\
 & +\E\int_{0}^{t}\vp^{\ve}(X_{r}^{\ve})dr.
\end{align*}
By Lemma \ref{lem:H1-bound} and Lemma \ref{lem:H-bound} we conclude
\begin{align*}
 & \E t\vp^{\ve}(X_{t}^{\ve})+\E\int_{0}^{t}r\|\ve\D X_{r}^{\ve}+\div\phi(\nabla X_{r}^{\ve})\|_{H}^{2}dr\\
 & \le\ve C\left(\E\|x_{0}\|_{H_{0}^{1}}^{2}+1\right)+C\left(\E\|x_{0}\|_{H}^{2}+1\right).
\end{align*}
To prove \eqref{approx_strong_S} we proceed as above but applying Itô's formula for $\vp^{\ve,\l}(X_{t}^{\ve})$ instead of $t\vp^{\ve,\l}(X_{t}^{\ve})$.
\end{proof}

\begin{proof}
[Proof of Theorem \ref{thm:mc-pinned}:]\textbf{ Step 1: $x_{0}\in L^{2}(\O;H_{0}^{1})$}

For $\ve_{1},\ve_{2}>0$ let $X^{\ve_{1}},X^{\ve_{2}}$ be two solutions to \eqref{eq:vis_approx} with initial conditions $x_{0}^{1},x_{0}^{2}\in L^{2}(\O;H_{0}^{1})$ respectively. Itô's formula implies
\begin{align*}
 & e^{-Kt}\|X_{t}^{\ve_{1}}-X_{t}^{\ve_{2}}\|_{H}^{2}\\
 & =\|x_{0}^{1}-x_{0}^{2}\|_{H}^{2}\\
 & +\int_{0}^{t}2e^{-Kr}\ _{V^{*}}\<\ve_{1}\D X_{r}^{\ve_{1}}+\div\phi(\nabla X_{r}^{\ve_{1}})-\left(\ve_{2}\D X_{r}^{\ve_{2}}+\div\phi(\nabla X_{r}^{\ve_{2}})\right),X_{r}^{\ve_{1}}-X_{r}^{\ve_{2}}\>_{V}dr\\
 & +\int_{0}^{t}2e^{-Kr}(X_{r}^{\ve_{1}}-X_{r}^{\ve_{2}},B(X_{r}^{\ve_{1}})-B(X_{r}^{\ve_{2}}))_{H}dW\\
 & +\int_{0}^{t}e^{-Kr}\|B(X_{r}^{\ve_{1}})-B(X_{r}^{\ve_{2}})\|_{L_{2}(U,H)}^{2}dr-K\E\int_{0}^{t}e^{-Kr}\|X_{r}^{\ve_{1}}-X_{r}^{\ve_{2}}\|_{H}^{2}dr.
\end{align*}
Since
\[
\ _{V^{*}}\<\div\phi(\nabla X_{r}^{\ve_{1}})-\div\phi(\nabla X_{r}^{\ve_{2}}),X_{r}^{\ve_{1}}-X_{r}^{\ve_{2}}\>_{V}\le0
\]
and
\[
\ _{V^{*}}\<\ve_{1}\D X_{r}^{\ve_{1}}-\ve_{2}\D X_{r}^{\ve_{1}},X_{r}^{\ve_{1}}-X_{r}^{\ve_{2}}\>_{V}\le2(\ve_{1}+\ve_{2})(\|X_{r}^{\ve_{1}}\|_{H_{0}^{1}}^{2}+\|X_{r}^{\ve_{2}}\|_{H_{0}^{1}}^{2}),
\]
we obtain
\begin{align*}
 & e^{-Kt}\|X_{t}^{\ve_{1}}-X_{t}^{\ve_{2}}\|_{H}^{2}\\
 & \le\|x_{0}^{1}-x_{0}^{2}\|_{H}^{2}+4(\ve_{1}+\ve_{2})\int_{0}^{t}e^{-Kr}(\|X_{r}^{\ve_{1}}\|_{H_{0}^{1}}^{2}+\|X_{r}^{\ve_{2}}\|_{H_{0}^{1}}^{2})dr\\
 & +\int_{0}^{t}2e^{-Kr}(X_{r}^{\ve_{1}}-X_{r}^{\ve_{2}},B(X_{r}^{\ve_{1}})-B(X_{r}^{\ve_{2}}))_{H}dW\\
 & +\int_{0}^{t}e^{-Kr}\|B(X_{r}^{\ve_{1}})-B(X_{r}^{\ve_{2}})\|_{L_{2}(U,H)}^{2}dr-K\E\int_{0}^{t}e^{-Kr}\|X_{r}^{\ve_{1}}-X_{r}^{\ve_{2}}\|_{H}^{2}dr.
\end{align*}
Using the Burkholder-Davis-Gundy inequality, Lemma \ref{lem:H1-bound} and choosing $K$ large enough implies
\begin{align}
\E\sup_{t\in[0,T]}\|X_{t}^{\ve_{1}}-X_{t}^{\ve_{2}}\|_{H}^{2}\le & C\E\|x_{0}^{1}-x_{0}^{2}\|_{H}^{2}\label{eq:approx_stab}\\
 & +(\ve_{1}+\ve_{2})C\left(\E\|x_{0}^{1}\|_{H_{0}^{1}}^{2}+\E\|x_{0}^{2}\|_{H_{0}^{1}}^{2}+1\right).\nonumber 
\end{align}
Now considering $X^{\ve}$ to be a solution to \eqref{eq:vis_approx} with initial condition $x_{0}\in L^{2}(\O;H_{0}^{1})$ for all $\ve>0$ yields 
\[
X^{\ve}\to X\quad\text{in }L^{2}(\O;C([0,T];H)),
\]
for $\ve\to0$. Due to Lemma \ref{lem:H1-bound} we have 
\[
\E\sup_{t\in[0,T]}e^{-Ct}\|X_{t}\|_{H_{0}^{1}}^{2}\le C\left(\E\|x_{0}\|_{H_{0}^{1}}^{2}+1\right).
\]
For two initial conditions $x_{0}^{1},x_{0}^{2}\in H_{0}^{1}$ and respective limits $X^{1},X^{2}$, \eqref{eq:approx_stab} then yields 
\begin{align}
 & \E\sup_{t\in[0,T]}\|X_{t}^{1}-X_{t}^{2}\|_{H}^{2}\le C\E\|x_{0}^{1}-x_{0}^{2}\|_{H}^{2}.\label{eq:stab_1}
\end{align}
It remains to identify $X$ as a strong solution to \eqref{eq:MC-pinned}. By Lemma \ref{lem:strong_approx} we have $\ve\D X^{\ve}+\div\phi(\nabla X^{\ve})$ uniformly bounded in $L^{2}([0,T]\times\O;H)$. Hence, there is an $\eta\in L^{2}([0,T]\times\O;H)$ and we can choose a sequence $\ve_{n}$ such that
\[
\ve_{n}\D X^{\ve_{n}}+\div\phi(\nabla X^{\ve_{n}})\rightharpoonup\eta,\quad\text{in }L^{2}([0,T]\times\O;H).
\]
We now aim to identify $\eta\in-\partial\vp(X)$, $dt\otimes d\P$-almost everywhere. By the subgradient property
\[
(\partial\vp^{\ve}(X_{t}^{\ve}),z-X_{t}^{\ve})_{2}+\vp^{\ve}(X_{t}^{\ve})-\vp^{\ve}(z)\le0
\]
for all $z\in L^{2}$. Since $X_{t}^{\ve}\in H_{0}^{1}\cap H^{2}$ we have $\partial\vp^{\ve}(X_{t}^{\ve})=-\ve\D X^{\ve}-\div\phi(\nabla X^{\ve})\in L^{2}$, $dt\otimes d\P$-almost everywhere. Integration yields 
\[
\E\int_{0}^{T}\t\left[-(\ve\D X^{\ve}+\div\phi(\nabla X^{\ve}),z-X^{\ve})_{2}+\vp^{\ve}(X^{\ve})-\vp^{\ve}(z)\right]dt\le0
\]
for all $z\in L^{2}$ and all non-negative $\t\in L^{\infty}([0,T]\times\O)$. Taking the limit yields
\[
\E\int_{0}^{T}\t\left[(-\eta,z-X)_{2}+\vp(X)-\vp(z)\right]dt\le0
\]
and thus
\[
(-\eta,z-X)_{2}+\vp(X)-\vp(z)\le0,
\]
for all $z\in L^{2}(\mcO)$, $dt\otimes d\P$-almost everywhere. Thus, $\eta\in-\partial\vp(X)$, $dt\otimes d\P$-almost everywhere. Since $\eta\in H$ and $X\in H_{0}^{1}$ $dt\otimes d\P$-a.e. by \eqref{eq:H1-subgradient} we have $\eta=\div\phi(\nabla X)$.  It is now easy to deduce that $X$ is a strong solution to \eqref{eq:MC-pinned} and $X$ satisfies
\begin{align*}
 & \E t\vp(X_{t})+\E\int_{0}^{t}r\|\eta_{r}\|_{H}^{2}dr\le C\left(\E\|x_{0}\|_{H}^{2}+1\right).
\end{align*}
and
\begin{equation}
\E\vp(X_{t})+\E\int_{0}^{t}\|\eta_{r}\|_{H}^{2}dr\le\E\vp(x_{0})+C,\label{approx_strong_S-2}
\end{equation}
for some constant $C>0$ (independent of $x_{0}$).

\textbf{Step 2: $x_{0}\in L^{2}(\O;H)$ }with $\E\vp(x_{0})<\infty$

By \cite[Proposition 8.2]{BR13}, for $v\in H_{0}^{1}$ we have 
\begin{align*}
\vp(J^{\l}v) & =\int_{\mcO}\psi(\nabla J^{\l}v)dx\\
 & \le\int_{\mcO}\psi(\nabla v)dx\\
 & =\vp(v),
\end{align*}
where $J^{\l}:=(1-\l\D)^{-1}$. Since $\vp$ is the lower-semicontinuous hull of $\vp_{|W_{0}^{1,1}\cap L^{2}}$ and thus of $\vp_{|H_{0}^{1}}$, for every $v\in H$ there is a sequence $v^{n}\in H_{0}^{1}$ with $v^{n}\to v$ in $H$ and $\vp(v^{n})\to\vp(v)$. Hence, 
\begin{align*}
\vp(J^{\l}v) & \le\vp(v),
\end{align*}
for all $v\in H$. We set $x_{0}^{n}:=J^{\frac{1}{n}}x_{0}$ and obtain 
\begin{align}
\E\vp(x_{0}^{n})+\E\|x_{0}^{n}\|_{2}^{2} & \le\E\vp(x_{0})+\E\|x_{0}\|_{2}^{2}<\infty.\label{eq:phi_ic_approx}
\end{align}
For $n,m>0$ let $X^{n},X^{m}$ be two solutions to \eqref{eq:MC-pinned} with $x_{0}=x_{0}^{n},x_{0}^{m}$ respectively as constructed in Step 1. From \eqref{eq:stab_1} we obtain
\begin{align}
 & \E\sup_{t\in[0,T]}\|X_{t}^{n}-X_{t}^{m}\|_{H}^{2}\le C\E\|x_{0}^{n}-x_{0}^{m}\|_{H}^{2}\label{eq:stab_1-1}
\end{align}
and thus
\[
X^{n}\to X\quad\text{in }L^{2}(\O;C([0,T];H)),
\]
for $n\to\infty$. Moreover, we have the uniform estimates
\begin{align}
\E t\vp(X_{t}^{n})+\E\int_{0}^{t}r\|\eta_{r}^{n}\|_{H}^{2}dr & \le C\left(\E\|x_{0}^{n}\|_{H}^{2}+1\right)\label{eq:approx_ineq}\\
 & \le C\left(\E\|x_{0}\|_{H}^{2}+1\right)\nonumber 
\end{align}
and
\begin{align}
\E\vp(X_{t}^{n})+\E\int_{0}^{t}\|\eta_{r}^{n}\|_{H}^{2}dr & \le\E\vp(x_{0}^{n})+C,\label{eq:approx_ineq-1}\\
 & \le\E\vp(x_{0})+C.\nonumber 
\end{align}
This allows the extraction of a subsequence and an $\eta\in L^{2}([0,T]\times\O;H)$ such that
\[
\eta^{n}\rightharpoonup\eta,\quad\text{in }L^{2}([0,T]\times\O;H).
\]
We may identify $\eta\in-\partial\vp(X)$ as in Step 1: Since $\eta^{n}\in-\partial\vp(X^{n})$, $dt\otimes d\P$-a.e.
\[
(-\eta^{n},z-X^{n})_{2}+\vp(X^{n})-\vp(z)\le0
\]
for all $z\in L^{2}$. Integrating against a non-negative testfunction $\t\in L^{\infty}([0,T]\times\O)$ and taking the limit $n\to\infty$ yields
\[
\E\int_{0}^{T}\t\left[(-\eta,z-X)_{2}+\vp(X)-\vp(z)\right]dt\le0
\]
and thus $\eta\in-\partial\vp(X)$, $dt\otimes d\P$-almost everywhere. Note that $-\partial\vp(v)=\div\phi(\nabla v)$ is known only for $v\in H_{0}^{1}$ which does not apply do $X$ in general. Taking the limit $n\to\infty$ in \eqref{eq:approx_ineq}, \eqref{eq:approx_ineq-1} yields the claim.

\textbf{Step 3: $x_{0}\in L^{2}(\O;H)$}

Let \textbf{$x_{0}^{n}\in L^{2}(\O;H_{0}^{1})$ }with $x_{0}^{n}\to x_{0}$ in $L^{2}(\O;H)$, $\E\|x_{0}^{n}\|_{H}^{2}\le\E\|x_{0}\|_{H}^{2}$ and let $X^{n}$ be the corresponding strong solutions constructed in step one. By \eqref{eq:stab_1} we have
\[
\E\sup_{t\in[0,T]}\|X^{n}-X^{m}\|_{H}^{2}\le C\E\|x_{n}-x_{m}\|_{H}^{2}.
\]
Hence, $X^{n}\to X$ in $L^{2}(\O;C([0,T];H))$. Moreover, 
\begin{align*}
 & \E t\vp(X_{t}^{n})+\E\int_{0}^{t}r\|\eta_{r}^{n}\|_{H}^{2}dr\le C\left(\E\|x_{0}\|_{H}^{2}+1\right).
\end{align*}
Hence, there is a map $\eta$ with $\eta\in L^{2}([\tau,T]\times\O;H)$ such that
\[
\eta^{n}\rightharpoonup\eta,\quad\text{in }L^{2}([\tau,T]\times\O;H),
\]
for all $\tau>0$. We may then prove $\eta\in-\partial\vp(X)$ as in Step 1. Hence, $X$ is a generalized strong solution satisfying 
\begin{align*}
 & \E t\vp(X_{t})+\E\int_{0}^{t}r\|\eta_{r}\|_{H}^{2}dr\le C\left(\E\|x_{0}\|_{H}^{2}+1\right).
\end{align*}

\end{proof}

\section{Mean curvature flow with (periodic) homogeneous normal noise \label{sec:mc-periodic}}

In this section we consider the SPDE
\begin{align}
dX_{t} & =\frac{\partial_{x}^{2}X_{t}}{1+(\partial_{x}X_{t})^{2}}dt+\a\sqrt{1+(\partial_{x}X_{t})^{2}}\circ d\b_{t}\label{eq:mc-normal}\\
X_{0} & =x_{0},\nonumber 
\end{align}
with periodic boundary conditions on $\mcO=(0,1)$ (i.e. $d=1$), $\b$ being a standard real-valued Brownian motion and $\a\le\sqrt{2}$. Informally rewriting the Stratonovich formulation of \eqref{eq:mc-normal} in Itô form as in \cite{ESR12} leads to the SPDE 
\begin{align}
dX_{t} & =\frac{\a^{2}}{2}\partial_{x}^{2}X_{t}dt+(1-\frac{\a^{2}}{2})\frac{\partial_{x}^{2}X_{t}}{1+(\partial_{x}X_{t})^{2}}dt+\a\sqrt{1+(\partial_{x}X_{t})^{2}}d\b_{t}\label{eq:mc-normal-ito}
\end{align}
with periodic boundary conditions on $\mcO=(0,1)$.  Let 
\begin{align*}
\psi(r) & =(1-\frac{\a^{2}}{2})\left(r\arctan(r)-\frac{1}{2}\log(r^{2}+1)\right)\\
\phi(r) & =\dot{\psi}(r)=(1-\frac{\a^{2}}{2})\arctan(r).
\end{align*}
For $v\in L^{2}(0,1)$ we define $v^{\bot}(x):=v(1-x)$. We then set 
\[
\vp(v):=\begin{cases}
\int_{\mcO}\psi(Dv)dx+\frac{1}{2}\int_{\partial\mcO}|[v-v^{\bot}]|H^{d-1}(dx) & \text{if }v\in(L^{2}\cap BV)(\mcO)\\
+\infty & \text{if }v\in(L^{2}\setminus BV)(\mcO),
\end{cases}
\]
where $\int_{\mcO}\psi(Dv)dx$ is defined as in Section \ref{sec:mc-pinned}. Since $\mcO=(0,1)$ we have
\[
\vp(v):=\int_{(0,1)}\psi(Dv)dx+|v(1)-v(0)|\quad\text{for }v\in(L^{2}\cap BV)(\mcO).
\]
In the following, for $p\ge1$ let 
\begin{align*}
W_{per}^{1,p}(\mcO) & :=\{f\in W^{1,p}(\mcO)|f(0)=f(1)\}\\
W_{per}^{2,p}(\mcO) & :=\{f\in(W^{2,p}\cap W_{per}^{1,p})(\mcO)|\partial_{x}f\in W_{per}^{1,p}\}.
\end{align*}
Moreover, let $H_{per}^{1}=W_{per}^{1,2}$, $H_{per}^{2}=W_{per}^{2,2}$, $H=L^{2}(\mcO)$. Then $\vp$ is the lower-semicontinuous envelope on $L^{2}$ of $\vp$ restricted to $W_{per}^{1,1}$ (cf. Appendix \ref{sec:periodic_relaxation}), i.e. of 
\[
\vp_{|W_{per}^{1,1}}(v)=\int_{\mcO}\psi(\partial_{x}v)dx,\quad v\in W_{per}^{1,1}\cap L^{2}.
\]
It is easy to see that $\vp_{|H_{per}^{1}}$ is Gateaux differentiable with
\[
D\vp_{|H_{per}^{1}}(v)(h)=\int_{\mcO}\phi(\partial_{x}v)\partial_{x}hdx.
\]
Since $\vp_{|W_{per}^{1,1}}$ is continuous on $W_{per}^{1,1}$ it is easy to see that $\vp$ is the lower-semicontinuous hull of $\vp_{|H_{per}^{1}}$ on $L^{2}$. This implies
\[
\partial\vp(u)=-\partial_{x}\phi(\partial_{x}u)=-(1-\frac{\a^{2}}{2})\frac{\partial_{x}^{2}u}{1+(\partial_{x}u)^{2}},\quad\text{for }u\in H_{per}^{2}.
\]
For $v\in H^{1}$ we define 
\[
B(v):=\a\sqrt{1+(\partial_{x}v)^{2}}.
\]
Hence, \eqref{eq:mc-normal} may be rewritten in the form
\begin{align}
dX_{t} & \in\frac{\a^{2}}{2}\partial_{x}^{2}X_{t}dt-\partial\vp(X_{t})dt+B(X_{t})d\b_{t},\label{eq:MC-pinned-2-1}\\
X_{0} & =x_{0}.\nonumber 
\end{align}
Due to the irregularity of the diffusion coefficients $B$ it does not seem possible to establish the existence of (generalized) strong solution as considered in Section \ref{sec:mc-pinned}. Instead, we introduce a notion of stochastic variational inequalities for \eqref{eq:MC-pinned-2-1}.

For regular initial data $x_{0}\in H_{per}^{1}$, the existence and uniqueness of variational solutions to \eqref{eq:mc-normal} has been shown in \cite{ESR12} (cf. also \cite{GT11} for multivalued generalizations). For general initial conditions $x_{0}\in L^{2}$ solutions have been constructed in \cite{ESR12} in a limiting sense. We now define what we mean by a solution to \eqref{eq:mc-normal}: 
\begin{defn}
\label{def:varn_periodic}Let $x_{0}\in L^{2}(\Omega;H).$ An $\mcF_{t}$-adapted process $X\in C([0,T];L^{2}(\Omega;H))$ is said to be an SVI solution to \eqref{eq:mc-normal} if there is an $\eta\in L^{2}([\tau,T]\times\Omega;H)$, $\forall\tau>0$ such that
\begin{enumerate}
\item {[}Regularity{]} 
\begin{align*}
\vp(X) & \in L^{1}([0,T]\times\Omega).
\end{align*}

\item {[}Subgradient property{]}
\[
\eta\in-\partial\vp(X),\quad dt\otimes d\P-a.e..
\]

\item {[}Stochastic variational inequality{]} For each $\mcF_{t}$-progressively measurable process $G\in L^{2}([0,T]\times\O;H)$ and each $\mcF_{t}$-adapted $H$-valued process $Z$ with $\P$-a.s. continuous sample paths such that $Z\in L^{2}([0,T]\times\O;H_{per}^{2})$ and solving the equation
\[
Z_{t}-Z_{0}=\int_{0}^{t}G_{s}ds+\int_{0}^{t}Z_{s}dW_{s},\quad\forall t\in[0,T]
\]
we have
\begin{align}
\E\|X_{t}-Z_{t}\|_{H}^{2} & \le\E\|X_{\tau}-Z_{\tau}\|_{H}^{2}+2\int_{\tau}^{t}(\eta_{r}-G_{r},X_{r}-Z_{r})_{2}dr\label{eq:var_ineq}\\
 & +\a^{2}\E\int_{\tau}^{t}(\partial_{x}^{2}Z_{r},X_{r}-Z_{r})_{2}dr,\quad\forall\tau>0.\nonumber 
\end{align}

\end{enumerate}
\end{defn}
\begin{rem}
\label{rmk:varn_sol}If $(X,\eta)$ is a generalized strong solution (defined analogously to Definition \ref{def:strong_soln})  to \eqref{eq:mc-normal} satisfying $\vp(X)\in L^{1}([0,T]\times\Omega)$ then $(X,\eta)$ is an SVI solution to \eqref{eq:mc-normal}.\end{rem}
\begin{proof}
Definition \ref{def:varn_periodic} (i),(ii) are satisfied by assumption. For (iii): Let $Z\in L^{2}([0,T]\times\O;H_{per}^{2})$ be a solution to 
\begin{align*}
dZ_{t} & =G_{t}dt+\a\sqrt{1+(\partial_{x}Z_{t})^{2}}d\b_{t}\\
 & =G_{t}dt+B(Z_{t})d\b_{t}
\end{align*}
for some $G\in L^{2}([0,T]\times\O;H)$. Then Itô's formula implies
\begin{align*}
\E\|X_{t}-Z_{t}\|_{H}^{2}= & \E\|X_{\tau}-Z_{\tau}\|_{H}^{2}+\a^{2}\E\int_{\tau}^{t}(\partial_{x}^{2}X_{r},X_{r}-Z_{r})_{2}dr\\
 & +2\E\int_{\tau}^{t}(\eta_{r}-G_{r},X_{r}-Z_{r})_{2}dr\\
 & +\a^{2}\E\int_{\tau}^{t}\|\sqrt{1+(\partial_{x}X_{r})^{2}}-\sqrt{1+(\partial_{x}Z_{r})^{2}}\|_{2}^{2}dr\quad\forall\tau>0.
\end{align*}
We note that
\begin{align*}
 & \a^{2}\|\sqrt{1+(\partial_{x}X_{r})^{2}}-\sqrt{1+(\partial_{x}Z_{r})^{2}}\|_{2}^{2}\\
 & \le\a^{2}\|\partial_{x}X_{r}-\partial_{x}Z_{r}\|_{2}^{2}\\
 & =-\a^{2}(\partial_{x}^{2}X_{r},X_{r}-Z_{r})_{2}+\a^{2}(\partial_{x}^{2}Z_{r},X_{r}-Z_{r})_{2},\quad dt\otimes d\P-\text{a.e.}
\end{align*}
and thus
\begin{align*}
\E\|X_{t}-Z_{t}\|_{H}^{2}\le & \E\|X_{\tau}-Z_{\tau}\|_{H}^{2}+2\E\int_{\tau}^{t}(\eta_{r}-G_{r},X_{r}-Z_{r})_{2}dr\\
 & +\a^{2}\E\int_{\tau}^{t}(\partial_{x}^{2}Z_{r},X_{r}-Z_{r})_{2}dr.
\end{align*}
In conclusion, each strong solution to \eqref{eq:mc-normal} is an SVI solution to \eqref{eq:mc-normal}. 
\end{proof}
The main result of the current section is the proof of well-posedness of \eqref{eq:mc-normal} in the sense of Definition \ref{def:varn_periodic}:
\begin{thm}
\label{thm:periodic_main}Let $x_{0}\in L^{2}(\Omega;H).$ Then there is a unique SVI solution $(X,\eta)$ to \eqref{eq:mc-normal} in the sense of Definition \ref{def:varn_periodic} satisfying
\begin{align*}
 & \E t\vp(X_{t})+\E\int_{0}^{t}r\|\eta_{r}\|_{H}^{2}dr\le C\left(\E\|x_{0}\|_{H}^{2}+1\right).
\end{align*}
In addition, if $\E\vp(x_{0})<\infty$ then
\begin{align*}
\E\vp(X_{t})+\E\int_{0}^{t}\|\eta_{r}\|_{H}^{2}dr & \le\E\vp(x_{0})+C.
\end{align*}
In particular, $\eta\in L^{2}([0,T]\times\Omega;H)$ and we may take $\tau=0$ in \eqref{eq:var_ineq}. 

For two SVI solutions $(X,\eta)$, $(Y,\z)$ with initial conditions $x_{0},y_{0}\in L^{2}(\Omega;H)$ we have
\[
\E\|X_{t}-Y_{t}\|_{H}^{2}\le\E\|x_{0}-y_{0}\|_{H}^{2},\quad\forall t\ge0.
\]

\end{thm}
For notational convenience we introduce the following semi-norm on $H_{per}^{1}$
\[
\|v\|_{H_{per,0}^{1}}:=\|\partial_{x}v\|_{2}.
\]
We note
\begin{align}
\|B(v)\|_{L_{2}(\R;H)}^{2} & =\a^{2}\left\Vert \sqrt{1+(\partial_{x}v)^{2}}\right\Vert _{H}^{2}=\a^{2}\int_{\mcO}1+(\partial_{x}v)^{2}dx,\quad\forall v\in H^{1}.\label{eq:periodic-B-growth}
\end{align}
and
\begin{align}
\|B(v)\|_{L_{2}(\R;H_{per,0}^{1})}^{2} & =\a^{2}\left\Vert \partial_{x}\sqrt{1+(\partial_{x}v)^{2}}\right\Vert _{H}^{2}\nonumber \\
 & =\a^{2}\int_{\mcO}\frac{(\partial_{x}v)^{2}(\partial_{x}^{2}v)^{2}}{1+(\partial_{x}v)^{2}}dx\label{eq:periodic-B-growth-strong}\\
 & \le\a^{2}\int_{\mcO}(\partial_{x}^{2}v)^{2}dx,\quad\forall v\in H^{2}.\nonumber 
\end{align}
Moreover,
\begin{align}
\|B(v)-B(w)\|_{L_{2}(\R;H)}^{2} & =\a^{2}\left\Vert \sqrt{1+(\partial_{x}v)^{2}}-\sqrt{1+(\partial_{x}w)^{2}}\right\Vert _{H}^{2}\label{eq:periodic-B-Lip}\\
 & \le\a^{2}\int_{\mcO}(\partial_{x}v-\partial_{x}w)^{2}dx,\quad\forall v,w\in H^{1}.\nonumber 
\end{align}
Some parts of the proof of Theorem \ref{thm:periodic_main} are analogous to the proof of Theorem \ref{thm:mc-pinned}. In this case we will restrict to short comments on the required modifications. The proof proceeds by considering vanishing viscosity approximations and regularizations in the initial condition. We shall first consider the case $x_{0}\in L^{2}(\O;H_{per}^{1})$ and
\begin{align}
dX_{t}^{\ve} & =\ve\partial_{x}^{2}X_{t}^{\ve}dt+(1-\frac{\a^{2}}{2})\frac{\partial_{x}^{2}X_{t}^{\ve}}{1+(\partial_{x}X_{t}^{\ve})^{2}}dt+\a\sqrt{1+(\partial_{x}X_{t}^{\ve})^{2}}\circ d\b_{t}\label{eq:mc-normal-visc-1}\\
X_{0}^{\ve} & =x_{0},\nonumber 
\end{align}
for $\ve\ge0$. Correspondingly, we set
\[
\vp^{\ve}(v)=\frac{\ve}{2}\int|\partial_{x}v|^{2}dx+\int\psi(\partial_{x}v)dx,\quad\text{for }v\in H_{per}^{1}.
\]
The variational formulation of \eqref{eq:mc-normal-visc-1} is based on the Gelfand triple 
\[
H_{per}^{1}\hookrightarrow L^{2}\hookrightarrow(H_{per}^{1})^{*}
\]
 and the variational operator
\[
_{(H_{per}^{1})^{*}}\<A^{\ve}(v),w\>_{H_{per}^{1}}:=-\ve\int_{\mcO}\partial_{x}v\partial_{x}wdx-\int_{\mcO}\phi(\partial_{x}v)\partial_{x}wdx,\quad\text{for }v,w\in H_{per}^{1}.
\]
By \cite{PR07} there is a unique solution to \eqref{eq:mc-normal-visc-1} in the sense of a variational solution $X^{\ve}\in L^{2}(\Omega;C([0,T];H))\cap L^{2}([0,T]\times\Omega;H_{per}^{1})$ to
\begin{align}
dX_{t}^{\ve}= & \ve\partial_{x}^{2}X_{t}^{\ve}dt+\frac{\a^{2}}{2}\partial_{x}^{2}X_{t}^{\ve}dt+(1-\frac{\a^{2}}{2})\frac{\partial_{x}^{2}X_{t}^{\ve}}{1+(\partial_{x}X_{t}^{\ve})^{2}}dt\label{eq:mc-normal-ito-visc}\\
 & +\a\sqrt{1+(\partial_{x}X_{t}^{\ve})^{2}}d\b_{t}.\nonumber 
\end{align}

\begin{lem}
\label{lem:H1-bound-1}For each $\ve>0$ we have $X^{\ve}\in L^{2}([0,T]\times\Omega;H_{per}^{2})$ and
\[
\E\sup_{t\in[0,T]}e^{-Ct}\|X_{t}^{\ve}\|_{H_{per}^{1}}^{2}+2\ve\int_{0}^{t}\E e^{-Cr}\|\partial_{x}^{2}X_{r}^{\ve}\|_{2}^{2}dr\le C\left(\E\|x_{0}\|_{H_{per}^{1}}^{2}+1\right),
\]
for some constant $C$ independent of $\ve>0.$ \end{lem}
\begin{proof}
As in Lemma \ref{lem:H1-bound} we may argue via Galerkin approximations $X^{n}$, where $(e_{i})_{i=1}^{\infty}$ now is an orthonormal basis of the periodic Laplacian $\partial_{x}^{2}$ on $L^{2}(\mcO)$. First note
\begin{align*}
\frac{\a^{2}}{2}\partial_{x}^{2}v+(1-\frac{\a^{2}}{2})\frac{\partial_{x}^{2}v}{1+(\partial_{x}v)^{2}} & =\frac{\a^{2}}{2}\frac{\partial_{x}^{2}v(1+(\partial_{x}v)^{2})}{1+(\partial_{x}v)^{2}}+(1-\frac{\a^{2}}{2})\frac{\partial_{x}^{2}v}{1+(\partial_{x}v)^{2}}\\
 & =\frac{\a^{2}}{2}\frac{\partial_{x}^{2}v(\partial_{x}v)^{2}}{1+(\partial_{x}v)^{2}}+\frac{\partial_{x}^{2}v}{1+(\partial_{x}v)^{2}}.
\end{align*}
Hence,
\begin{align*}
 & 2(-\partial\vp^{\ve}(v),v)_{H_{per,0}^{1}}+\|B(v)\|_{H_{per,0}^{1}}^{2}\\
= & -(2\ve\partial_{x}^{2}v+\a^{2}\partial_{x}^{2}v+(2-\a^{2})\frac{\partial_{x}^{2}v}{1+(\partial_{x}v)^{2}},\partial_{x}^{2}v)_{2}+\|\partial_{x}\a\sqrt{1+(\partial_{x}v)^{2}}\|_{2}^{2}\\
= & -(2\ve\partial_{x}^{2}v+\a^{2}\frac{\partial_{x}^{2}v(\partial_{x}v)^{2}}{1+(\partial_{x}v)^{2}}+2\frac{\partial_{x}^{2}v}{1+(\partial_{x}v)^{2}},\partial_{x}^{2}v)_{2}+\|\partial_{x}\a\sqrt{1+(\partial_{x}v)^{2}}\|_{2}^{2}\\
= & -2\ve\int_{\mcO}(\partial_{x}^{2}v)^{2}dx-2\int_{\mcO}\frac{(\partial_{x}^{2}v)^{2}}{1+(\partial_{x}v)^{2}}dx-\a^{2}\int_{\mcO}\frac{(\partial_{x}v)^{2}(\partial_{x}^{2}v)^{2}}{1+(\partial_{x}v)^{2}}dx+\a^{2}\int_{\mcO}\frac{(\partial_{x}v)^{2}(\partial_{x}^{2}v)^{2}}{1+(\partial_{x}v)^{2}}dx\\
= & -2\ve\int_{\mcO}(\partial_{x}^{2}v)^{2}dx-2\int_{\mcO}\frac{(\partial_{x}^{2}v)^{2}}{1+(\partial_{x}v)^{2}}dx\\
\le & -2\ve\int_{\mcO}(\partial_{x}^{2}v)^{2}dx,
\end{align*}
for all $v\in H_{per}^{2}$. Itô's formula thus implies
\begin{align*}
\|X_{t}^{n}\|_{H_{per,0}^{1}}^{2}= & \|P^{n}x_{0}\|_{H_{per,0}^{1}}^{2}\\
 & +2\int_{0}^{t}(\ve P^{n}\partial_{x}^{2}X_{r}^{n}+P^{n}\a^{2}\partial_{x}^{2}X_{r}^{n}+(2-\a^{2})P^{n}\frac{\partial_{x}^{2}X_{r}^{n}}{1+(\partial_{x}X_{r}^{n})^{2}},X_{r}^{n})_{H_{per,0}^{1}}dr\\
 & +2\int_{0}^{t}(P^{n}B(X_{r}^{n}),X_{r}^{n})_{H_{per,0}^{1}}dW_{r}+\int_{0}^{t}\|P^{n}B(X_{r}^{n})\|_{L_{2}(\R,H_{per,0}^{1})}^{2}dr\\
\le & \|x_{0}\|_{H_{per,0}^{1}}^{2}-2\ve\int_{0}^{t}\|\partial_{x}^{2}X_{r}^{n}\|_{2}^{2}dr+\int_{0}^{t}(B(X_{r}^{n}),X_{r}^{n})_{H_{per,0}^{1}}dW_{r}.
\end{align*}
Observing
\begin{align*}
\|X_{t}^{n}\|_{H}^{2}= & \|P^{n}x_{0}\|_{H}^{2}\\
 & +2\int_{0}^{t}(\ve P^{n}\partial_{x}^{2}X_{r}^{n}+P^{n}\a^{2}\partial_{x}^{2}X_{r}^{n}+(2-\a^{2})P^{n}\frac{\partial_{x}^{2}X_{r}^{n}}{1+(\partial_{x}X_{r}^{n})^{2}},X_{r}^{n})_{H}dr\\
 & +2\int_{0}^{t}(P^{n}B(X_{r}^{n}),X_{r}^{n})_{H}dW_{r}+\int_{0}^{t}\|P^{n}B(X_{r}^{n})\|_{L_{2}(\R;H)}^{2}dr\\
\le & \|x_{0}\|_{H}^{2}+\int_{0}^{t}(B(X_{r}^{n}),X_{r}^{n})_{H}dW_{r},
\end{align*}
the proof may be completed as in Lemma \ref{lem:H1-bound}.\end{proof}
\begin{lem}
\label{lem:H-bound-2}For each $\ve>0$ we have $\vp^{\ve}(X^{\ve})\in L^{1}([0,T]\times\Omega)$ with
\begin{align*}
\E\int_{0}^{T}e^{-Cr}\vp^{\ve}(X_{r}^{\ve})dr & \le C\left(\E\|x_{0}\|_{H}^{2}+1\right),
\end{align*}
for some constant $C$ independent of $\ve>0.$ \end{lem}
\begin{proof}
By Itô's formula and \eqref{eq:coerc_of_Aeps} we have 
\begin{align*}
\E e^{-Kt}\|X_{t}^{\ve}\|_{H}^{2}= & \E\|x_{0}\|_{H}^{2}+2\E\int_{0}^{t}e^{-Kr}(\frac{\a^{2}}{2}\partial_{x}^{2}X_{r}^{\ve}+A^{\ve}(X_{r}^{\ve}),X_{r}^{\ve})_{H}dr\\
 & +2\E\int_{0}^{t}e^{-Kr}\|B(X_{r}^{\ve})\|_{L_{2}(\R,H)}^{2}dr-K\int_{0}^{t}e^{-Kr}\|X_{r}^{\ve}\|_{H}^{2}dr\\
\le & \E\|x_{0}\|_{H}^{2}-2\E\int_{0}^{t}e^{-Kr}\vp^{\ve}(X_{r}^{\ve})+Ce^{-Kr}\|X_{r}^{\ve}\|_{H}^{2}dr\\
 & -K\int_{0}^{t}e^{-Kr}\|X_{r}^{\ve}\|_{H}^{2}dr.
\end{align*}
Choosing $K$ large enough yields the claim.\end{proof}
\begin{lem}
\label{lem:strong_approx-1}Let $x_{0}\in L^{2}(\O;H_{per}^{1})$. For all $\ve>0$ we have
\begin{align}
 & \E t\vp^{\ve}(X_{t}^{\ve})+\E\int_{0}^{t}r\|\ve\partial_{x}^{2}X_{r}^{\ve}+\partial_{x}\phi(\partial_{x}X_{r}^{\ve})\|_{H}^{2}dr\le C\left(\E\|x_{0}\|_{H}^{2}+1\right)\label{eq:approx_strong_H-1}
\end{align}
and
\begin{equation}
\E\vp^{\ve}(X_{t}^{\ve})+\E\int_{0}^{t}\|\ve\partial_{x}^{2}X_{r}^{\ve}+\partial_{x}\phi(\partial_{x}X_{r}^{\ve})\|_{H}^{2}dr\le\E\vp^{\ve}(x_{0})+C,\label{approx_strong_S-5}
\end{equation}
for some constant $C>0$.\end{lem}
\begin{proof}
We first prove \eqref{eq:approx_strong_H-1}: Let $J^{\l}$ be the resolvent of $-\partial_{x}^{2}$ on $L^{2}(\mcO)$ with domain $\mcD(-\partial_{x}^{2})=H_{per}^{2}(\mcO)$. As in the proof of Lemma \ref{lem:strong_approx} we obtain
\begin{align}
 & \E t\vp^{\ve,\l}(X_{t}^{\ve})\nonumber \\
= & -\E\int_{0}^{t}r(\ve\partial_{x}^{2}J^{\l}X_{r}^{\ve}+\partial_{x}\phi(\partial_{x}J^{\l}X_{r}^{\ve}),\ve J^{\l}\partial_{x}^{2}X_{r}^{\ve}+J^{\l}\partial_{x}\phi(\partial_{x}X_{r}^{\ve}))_{H}dr\nonumber \\
 & -\E\int_{0}^{t}r(\ve\partial_{x}^{2}J^{\l}X_{r}^{\ve}+\partial_{x}\phi(\partial_{x}J^{\l}X_{r}^{\ve}),\frac{\a^{2}}{2}J^{\l}\partial_{x}^{2}X_{r}^{\ve})_{H}dr\label{eq:periodic_strong_test}\\
 & +\frac{\ve\a^{2}}{2}\int_{0}^{t}r\int_{\mcO}\left(\partial_{x}J^{\l}\sqrt{1+(\partial_{x}X_{r}^{\ve})^{2}}\right)^{2}dxdr\nonumber \\
 & +\frac{\a^{2}}{2}\int_{0}^{t}r\int_{\mcO}\dot{\phi}(\partial_{x}J^{\l}X_{r}^{\ve})\left(\partial_{x}J^{\l}\sqrt{1+(\partial_{x}X_{r}^{\ve})^{2}}\right)^{2}dxdr\nonumber \\
 & +\E\int_{0}^{t}\vp^{\ve,\l}(X_{r}^{\ve})dr.\nonumber 
\end{align}
 We first note that
\begin{align*}
 & \int_{\mcO}(\partial_{x}J^{\l}\sqrt{1+(\partial_{x}v)^{2}})^{2}dx\\
 & \le\int_{\mcO}(\partial_{x}\sqrt{1+(\partial_{x}v)^{2}})^{2}dx\\
 & =\int_{\mcO}\frac{(\partial_{x}v)^{2}(\partial_{x}^{2}v)^{2}}{1+(\partial_{x}v)^{2}}dx\\
 & \le\|\partial_{x}^{2}v\|_{2}^{2},
\end{align*}
for all $v\in H_{per}^{2}$. Moreover,
\begin{align*}
 & \dot{\phi}(\partial_{x}J^{\l}(v))(\partial_{x}J^{\l}\sqrt{1+(\partial_{x}v)^{2}})^{2}\\
 & =(1-\frac{\a^{2}}{2})\frac{(\partial_{x}J^{\l}\sqrt{1+(\partial_{x}v)^{2}})^{2}}{1+(\partial_{x}J^{\l}v)^{2}}\\
 & \le(1-\frac{\a^{2}}{2})(\partial_{x}J^{\l}\sqrt{1+(\partial_{x}v)^{2}})^{2},
\end{align*}
for all $v\in H_{per}^{2}$. Since $r\mapsto\sqrt{1+r^{2}}$ is Lipschitz we have 
\[
\sqrt{1+(\partial_{x}v)^{2}}\in H_{per}^{1}
\]
 for $v\in H_{per}^{2}$ and thus (cf. \cite[Theorem 2.13]{MR92})
\[
J^{\l}\sqrt{1+(\partial_{x}v)^{2}}\to\sqrt{1+(\partial_{x}v)^{2}}\quad\text{in }H_{per}^{1},
\]
for $\l\to0$. Moreover, $J^{\l}v\to v$ in $H_{per}^{2}$ for $v\in H_{per}^{2}$. Thus, $\partial_{x}J^{\l}v\to\partial_{x}v$ in $H_{per}^{1}\subseteq L^{\infty}$. Since $\dot{\phi}$ is Lipschitz we have $\dot{\phi}(\partial_{x}J^{\l}v)\to\dot{\phi}(\partial_{x}v)$ in $L^{2}([0,T]\times\O;L^{\infty})$ for all $v\in L^{2}([0,T]\times\O;H_{per}^{2})$. Hence,
\begin{align*}
 & \lim_{\l}\E\int_{0}^{t}r\int_{\mcO}\dot{\phi}(\partial_{x}J^{\l}v_{r})\left(\partial_{x}J^{\l}\sqrt{1+(\partial_{x}v_{r})^{2}}\right)^{2}dxdr\\
 & =\E\int_{0}^{t}r\int_{\mcO}\frac{(\partial_{x}\sqrt{1+(\partial_{x}v_{r})^{2}})^{2}}{1+(\partial_{x}v_{r})^{2}}dxdr\\
 & =\E\int_{0}^{t}r\int_{\mcO}\frac{(\partial_{x}v_{r})^{2}(\partial_{x}^{2}v_{r})^{2}}{(1+(\partial_{x}v_{r})^{2})^{2}}dxdr\\
 & \le\E\int_{0}^{t}r\int_{\mcO}\frac{(\partial_{x}^{2}v_{r})^{2}}{1+(\partial_{x}v_{r})^{2}}dxdr,
\end{align*}
for all $v\in L^{2}([0,T]\times\O;H_{per}^{2})$. Taking the limit $\l\to0$ in the first term on the right hand side of \eqref{eq:periodic_strong_test} can be justified as in Lemma \ref{lem:strong_approx} . This yields
\begin{align*}
\E t\vp^{\ve}(X_{t}^{\ve})\le & -\E\int_{0}^{t}r\|\ve\partial_{x}^{2}X_{r}^{\ve}+\partial_{x}\phi(\partial_{x}X_{r}^{\ve})\|_{H}^{2}dr\\
 & -\frac{\ve\a^{2}}{2}\E\int_{0}^{t}r\|\partial_{x}^{2}X_{r}^{\ve}\|_{H}^{2}dr-\frac{\a^{2}}{2}\E\int_{0}^{t}r\int_{\mcO}\frac{(\partial_{x}^{2}X_{r}^{\ve})^{2}}{1+(\partial_{x}X_{r}^{\ve})^{2}}dxdr\\
 & +\frac{\ve\a^{2}}{2}\E\int_{0}^{t}r\|\partial_{x}^{2}X_{r}^{\ve}\|_{H}^{2}dr+\frac{\a^{2}}{2}\E\int_{0}^{t}r\int_{\mcO}\frac{(\partial_{x}^{2}X_{r}^{\ve})^{2}}{1+(\partial_{x}X_{r}^{\ve})^{2}}dxdr\\
 & +\E\int_{0}^{t}\vp^{\ve}(X_{r}^{\ve})dr\\
= & -\E\int_{0}^{t}r\|\ve\partial_{x}^{2}X_{r}^{\ve}+\partial_{x}\phi(\partial_{x}X_{r}^{\ve})\|_{H}^{2}dr+\E\int_{0}^{t}\vp^{\ve}(X_{r}^{\ve})dr.
\end{align*}
By Lemma \ref{lem:H-bound-2} we conclude
\begin{align*}
 & \E t\vp^{\ve}(X_{t}^{\ve})+\E\int_{0}^{t}r\|\ve\partial_{x}^{2}X_{r}^{\ve}+\partial_{x}\phi(\partial_{x}X_{r}^{\ve})\|_{H}^{2}dr\le C\left(\E\|x_{0}\|_{H}^{2}+1\right).
\end{align*}
To prove \eqref{approx_strong_S-5} we proceed as above but applying Itô's formula for $\vp^{\ve,\l}(X_{t}^{\ve})$ instead of $t\vp^{\ve,\l}(X_{t}^{\ve})$.
\end{proof}

\begin{proof}[Proof of Theorem \ref{thm:periodic_main}]
\textbf{Step 1:} Existence 

We start with the construction via an approximation of the initial condition. Let $x_{0}^{n}\in L^{2}(\O;H_{per}^{1})$ with $x_{0}^{n}\to x$ in $L^{2}(\O;H)$. By Lemma \ref{lem:H1-bound-1} there are strong solutions $X^{\ve,n}$ to 
\begin{align*}
dX_{t}^{\ve,n} & =\ve\partial_{x}^{2}X_{t}^{\ve,n}dt+\frac{\a^{2}}{2}\partial_{x}^{2}X_{t}^{\ve,n}dt-\partial\vp(X_{t}^{\ve,n})dt+B(X_{t}^{\ve,n})d\b_{t}\\
X_{t}^{\ve,n} & =x_{0}^{n},
\end{align*}
satisfying
\[
\E\sup_{t\in[0,T]}\|X_{t}^{\ve,n}\|_{H_{per}^{1}}^{2}+\ve\E\int_{0}^{T}\|\partial_{x}^{2}X_{r}^{\ve,n}\|_{2}^{2}dr\le C(\E\|x_{0}^{n}\|_{H_{per}^{1}}^{2}+1).
\]
We will first prove convergence of $X^{\ve,n}$ for $\ve\to0.$ For $\ve_{1},\ve_{2}>0$ let $X^{\ve_{1}},X^{\ve_{2}}$ be two solutions to \eqref{eq:mc-normal-ito-visc} with initial conditions $x_{0}^{1},x_{0}^{2}\in L^{2}(\O;H_{per}^{1})$. Itô's formula, Lemma \ref{lem:H1-bound-1} and \eqref{eq:periodic-B-Lip} imply
\begin{align}
 & \E\|X_{t}^{\ve_{1}}-X_{t}^{\ve_{2}}\|_{H}^{2}\nonumber \\
 & =\E\|x_{0}^{1}-x_{0}^{2}\|_{H}^{2}\nonumber \\
 & +\E\int_{0}^{t}2(\ve_{1}\partial_{x}^{2}X_{r}^{\ve_{1}}-\ve_{2}\partial_{x}^{2}X_{r}^{\ve_{2}},X_{r}^{\ve_{1}}-X_{r}^{\ve_{2}})_{2}dr\nonumber \\
 & +\frac{\a^{2}}{2}\E\int_{0}^{t}2(\partial_{x}^{2}X_{r}^{\ve_{1}}-\partial_{x}^{2}X_{r}^{\ve_{1}},X_{r}^{\ve_{1}}-X_{r}^{\ve_{2}})_{2}dr\nonumber \\
 & +\E\int_{0}^{t}2(\partial_{x}\phi(\partial_{x}X_{r}^{\ve_{1}})-\partial_{x}\phi(\partial_{x}X_{r}^{\ve_{2}}),X_{r}^{\ve_{1}}-X_{r}^{\ve_{2}})_{2}dr\label{eq:stab}\\
 & +\E\int_{0}^{t}\|B(X_{r}^{\ve_{1}})-B(X_{r}^{\ve_{2}})\|_{L_{2}(\R,H)}^{2}dr\nonumber \\
 & \le\E\|x_{0}^{1}-x_{0}^{2}\|_{H}^{2}+4(\ve_{1}+\ve_{2})\E\int_{0}^{t}\|X_{r}^{\ve_{1}}\|_{H_{per}^{1}}^{2}+\|X_{r}^{\ve_{2}}\|_{H_{per}^{1}}^{2}dr\nonumber \\
 & -\a^{2}\E\int_{0}^{t}\int_{\mcO}(\partial_{x}X_{r}^{\ve_{1}}-\partial_{x}X_{r}^{\ve_{2}})^{2}dxdr+\a^{2}\E\int_{0}^{t}\int_{\mcO}(\partial_{x}X_{r}^{\ve_{1}}-\partial_{x}X_{r}^{\ve_{2}})^{2}dxdr\nonumber \\
 & \le\E\|x_{0}^{1}-x_{0}^{2}\|_{H}^{2}+(\ve_{1}+\ve_{2})C\left(\E\|x_{0}^{1}\|_{H_{per}^{1}}^{2}+\E\|x_{0}^{2}\|_{H_{per}^{1}}^{2}+1\right).\nonumber 
\end{align}
Hence,
\[
X^{\ve,n}\to X^{n}\quad\text{in }C([0,T];L^{2}(\O;H)),\quad\text{for }\ve\to0
\]
for some $\mcF_{t}$-adapted $X^{n}$. Due to \eqref{eq:stab} we obtain
\begin{align*}
\E\|X_{t}^{n}-X_{t}^{m}\|_{H}^{2} & \le\E\|x_{0}^{n}-x_{0}^{m}\|_{H}^{2}
\end{align*}
and thus
\[
X^{n}\to X\quad\text{in }C([0,T];L^{2}(\O;H)),\quad\text{for }n\to\infty
\]
for some $\mcF_{t}$-adapted $X$. We shall now prove that $X$ is an SVI solution to \eqref{eq:mc-normal}. By Lemma \ref{lem:strong_approx-1} we have
\begin{align}
\E t\vp^{\ve}(X_{t}^{\ve,n})+\E\int_{0}^{t}r\|\ve\partial_{x}^{2}X^{\ve,n}+\partial_{x}\phi(\partial_{x}X^{\ve,n})\|_{H}^{2}dr\le & C\left(\E\|x_{0}^{n}\|_{H}^{2}+1\right)\label{eq:approx_strong}
\end{align}
and thus there is a function $\eta$ and a sequence $\ve_{m}\to0$ such that for each $\tau>0$ 
\[
\ve_{m}\partial_{x}^{2}X^{\ve_{m},n}+\partial_{x}\phi(\partial_{x}X^{\ve_{m},n})\rightharpoonup\eta^{n},\quad\text{in }L^{2}([\tau,T]\times\O;H).
\]
We can prove $\eta^{n}\in-\partial\vp(X^{n})$ as in the proof of Theorem \ref{thm:mc-pinned}. Taking the limit in \eqref{eq:approx_strong} yields
\begin{align*}
 & \E t\vp(X_{t}^{n})+\E\int_{0}^{t}r\|\eta^{n}\|_{H}^{2}dr\le C\left(\E\|x_{0}^{n}\|_{H}^{2}+1\right).
\end{align*}
We can now argue as above to obtain the existence of an $\eta\in-\partial\vp(X)$ and a subsequence of $\eta^{n}$ (again denoted by $\eta^{n}$ for simplicity) such that
\[
\eta^{n}\rightharpoonup\eta,\quad\text{in }L^{2}([\tau,T]\times\O;H),
\]
for all $\tau>0$. As in Remark \ref{rmk:varn_sol} we have
\begin{align*}
\E\|X_{t}^{\ve,n}-Z_{t}\|_{H}^{2}\le & \E\|X_{\tau}^{\ve,n}-Z_{\tau}\|_{H}^{2}+2\E\int_{\tau}^{t}(-\partial\vp(X_{r}^{\ve,n})-G_{r},X_{r}^{\ve,n}-Z_{r})_{2}dr\\
 & +\a^{2}\E\int_{\tau}^{t}(\partial_{x}^{2}Z_{r},X_{r}^{\ve,n}-Z_{r})_{2}dr+2\ve\E\int_{\tau}^{t}(\partial_{x}^{2}X_{r}^{\ve,n},X_{r}^{\ve,n}-Z_{r})_{2}dr,
\end{align*}
for all $t\ge\tau>0$. Using Lemma \ref{lem:H1-bound-1} we observe
\begin{align*}
\ve\E\int_{\tau}^{t}(\partial_{x}^{2}X_{r}^{\ve,n},X_{r}^{\ve,n}-Z_{r})_{2}dr & \le\ve^{\frac{4}{3}}\E\int_{\tau}^{t}\|\partial_{x}^{2}X_{r}^{\ve,n}\|_{2}^{2}dr+\ve^{\frac{2}{3}}\E\int_{\tau}^{t}\|X_{r}^{\ve,n}-Z_{r}\|_{2}^{2}dr\\
 & \le\ve^{\frac{1}{3}}C(\E\|x_{0}^{n}\|_{H_{per}^{1}}^{2}+1)+\ve^{\frac{2}{3}}\E\int_{\tau}^{t}\|X_{r}^{\ve,n}-Z_{r}\|_{2}^{2}dr.
\end{align*}
We obtain
\begin{align*}
\E\|X_{t}^{\ve,n}-Z_{t}\|_{H}^{2}\le & \E\|X_{\tau}^{\ve,n}-Z_{\tau}\|_{H}^{2}+2\E\int_{\tau}^{t}(-\partial\vp(X_{r}^{\ve,n})-G_{r},X_{r}^{\ve,n}-Z_{r})_{2}dr\\
 & +\a^{2}\E\int_{\tau}^{t}(\partial_{x}^{2}Z_{r},X_{r}^{\ve,n}-Z_{r})_{2}dr\\
 & +2\ve^{\frac{1}{3}}C(\E\|x_{0}^{n}\|_{H_{per}^{1}}^{2}+1)+2\ve^{\frac{2}{3}}\E\int_{\tau}^{t}\|X_{r}^{\ve,n}-Z_{r}\|_{2}^{2}dr.
\end{align*}
Now we take $\ve\to0$ to obtain
\begin{align*}
\E\|X_{t}^{n}-Z_{t}\|_{H}^{2}\le & \E\|X_{\tau}^{n}-Z_{\tau}\|_{H}^{2}+2\E\int_{\tau}^{t}(\eta_{r}^{n}-G_{r},X_{r}^{n}-Z_{r})_{2}dr\\
 & +\a^{2}\E\int_{\tau}^{t}(\partial_{x}^{2}Z_{r},X_{r}^{n}-Z_{r})_{2}dr.
\end{align*}
Taking $n\to0$ yields the claim.

\textbf{Step 2: }Uniqueness 

Let $X$ be an SVI solution to \eqref{eq:mc-normal} with initial condition $x_{0}\in L^{2}(\O;H)$. Further let $y_{0}\in L^{2}(\O;H)$ and $y_{0}^{n}\in L^{2}(\O;H_{per}^{1}$) with $y_{0}^{n}\to y$ in $L^{2}(\O;H)$. Due to Lemma \ref{lem:strong_approx-1} there are strong solutions $Y^{\ve,n}$ to 
\begin{align*}
dY_{t}^{\ve,n} & =\ve\partial_{x}^{2}Y_{t}^{\ve,n}dt+\frac{\a^{2}}{2}\partial_{x}^{2}Y_{t}^{\ve,n}dt+(1-\frac{\a^{2}}{2})\frac{\partial_{x}^{2}Y_{t}^{\ve,n}}{1+(\partial_{x}Y_{t}^{\ve,n})^{2}}dt\\
 & +\a\sqrt{1+(\partial_{x}Y_{t}^{\ve,n})^{2}}d\b_{t}.\\
Y_{0}^{\ve,n} & =y_{0}^{n},
\end{align*}
satisfying
\[
\E\sup_{t\in[0,T]}\|Y_{t}^{\ve,n}\|_{H_{per}^{1}}^{2}+\ve\E\int_{0}^{T}\|\partial_{x}^{2}Y_{r}^{\ve,n}\|_{2}^{2}dr\le C(\E\|x_{0}^{n}\|_{H_{per}^{1}}^{2}+1).
\]
Using the variational inequality with $Z\equiv Y^{\ve,n}$ and 
\begin{align*}
G & =\ve\partial_{x}^{2}Y^{\ve,n}+\frac{\a^{2}}{2}\partial_{x}^{2}Y^{\ve,n}+\partial_{x}\phi(\partial_{x}Y^{\ve,n})
\end{align*}
we obtain
\begin{align*}
\E\|X_{t}-Y_{t}^{\ve,n}\|_{H}^{2} & \le\E\|X_{\tau}-Y_{\tau}^{\ve,n}\|_{H}^{2}\\
 & +2\E\int_{\tau}^{t}(\eta_{r}-\ve\partial_{x}^{2}Y_{r}^{\ve,n}-\frac{\a^{2}}{2}\partial_{x}^{2}Y_{r}^{\ve,n}-\partial_{x}\phi(\partial_{x}Y_{r}^{\ve,n}),X_{r}-Y_{r}^{\ve,n})_{2}dr\\
 & +\a^{2}\E\int_{\tau}^{t}(\partial_{x}^{2}Y_{r}^{\ve,n},X_{r}-Y_{r}^{\ve,n})_{2}dr\\
= & \E\|X_{\tau}-Y_{\tau}^{\ve,n}\|_{H}^{2}\\
 & +2\E\int_{\tau}^{t}(\eta_{r}-\partial_{x}\phi(\partial_{x}Y_{r}^{\ve,n}),X_{r}-Y_{r}^{\ve,n})_{2}dr\\
 & -\ve2\E\int_{\tau}^{t}(\partial_{x}^{2}Y_{r}^{\ve,n},X_{r}-Y_{r}^{\ve,n})_{2}dr,
\end{align*}
for all $t\ge\tau>0$. Since $-\partial_{x}\phi(\partial_{x}Y^{\ve,n})=\partial\vp(Y^{\ve,n})$ and $\eta\in-\partial\vp(X)$, $dt\otimes d\P$-a.e. we have
\begin{align*}
(\eta-\partial_{x}\phi(\partial_{x}Y^{\ve,n}),X-Y^{\ve,n})_{2} & =-(-\eta-\partial\vp(Y^{\ve,n}),X-Y^{\ve,n})_{2}\\
 & \le0,\quad dt\otimes d\P-\text{a.e..}
\end{align*}
Hence,
\begin{align*}
\E\|X_{t}-Y_{t}^{\ve,n}\|_{H}^{2} & \le\E\|X_{\tau}-Y_{\tau}^{\ve,n}\|_{H}^{2}-2\ve\E\int_{0}^{t}(\partial_{x}^{2}Y_{r}^{\ve,n}dt,X_{r}-Y_{r}^{\ve,n})_{2}dr.
\end{align*}
We further note
\begin{align*}
\ve(\partial_{x}^{2}Y_{r}^{\ve,n},X_{r}-Y_{r}^{\ve,n}) & \le\ve\|\partial_{x}^{2}Y_{r}^{\ve,n}\|_{2}\|X_{r}-Y_{r}^{\ve,n}\|_{2}\\
 & \le\ve^{\frac{4}{3}}\|\partial_{x}^{2}Y_{r}^{\ve,n}\|_{2}^{2}+\ve^{\frac{2}{3}}\|X_{r}-Y_{r}^{\ve,n}\|_{2}^{2}.
\end{align*}
Due to Lemma \ref{lem:strong_approx-1} this implies 
\[
2\ve\E\int_{0}^{t}(\partial_{x}^{2}Y_{r}^{\ve,n},X_{r}-Y_{r}^{\ve,n})_{2}dr\le\ve^{\frac{1}{4}}C(\E\|y_{0}^{n}\|_{H_{per}^{1}}^{2}+1)+2\ve^{\frac{2}{3}}\E\int_{0}^{t}\|X_{r}-Y_{r}^{\ve,n}\|_{2}^{2}dr.
\]
Thus
\begin{align*}
\E\|X_{t}-Y_{t}^{\ve,n}\|_{H}^{2} & \le\E\|X_{\tau}-Y_{\tau}^{\ve,n}\|_{H}^{2}\\
 & +\ve^{\frac{1}{4}}C(\E\|y_{0}^{n}\|_{H_{per}^{1}}^{2}+1)+2\ve^{\frac{2}{3}}\E\int_{0}^{t}\|X_{r}-Y_{r}^{\ve,n}\|_{2}^{2}dr.
\end{align*}
For $n\in N$ arbitrary, fixed we have seen in step one 
\[
Y^{\ve,n}\to Y^{n}\quad\text{in }C([0,T];L^{2}(\O;H))\quad\text{for }\ve\to0
\]
and
\[
Y^{n}\to Y\quad\text{in }C([0,T];L^{2}(\O;H)),\quad\text{for }n\to\infty,
\]
where $Y$ is an SVI solution to \eqref{eq:mc-normal} with initial condition $y_{0}$. We obtain
\begin{align*}
\sup_{t\in[0,T]}\E\|X_{t}-Y_{t}\|_{H}^{2} & \le\E\|X_{\tau}-Y_{\tau}\|_{H}^{2}.
\end{align*}
Now letting $\tau\to0$ yields
\begin{align*}
\sup_{t\in[0,T]}\E\|X_{t}-Y_{t}\|_{H}^{2} & \le\E\|x_{0}-y_{0}\|_{H}^{2}.
\end{align*}
In particular, choosing $y_{0}=x_{0}$ implies uniqueness of SVI solutions. 

\textbf{Step 3: $x_{0}\in L^{2}(\O;H)$ }with $\E\vp(x_{0})<\infty$

As in \eqref{eq:phi_ic_approx} we may choose the approximations $x_{0}^{n}\in L^{2}(\O;H_{0}^{1})\subseteq L^{2}(\O;H_{per}^{1})$ of $x_{0}$ considered in step one such that 
\begin{align*}
\E\vp(x_{0}^{n})+\E\|x_{0}^{n}\|_{2}^{2} & \le\E\vp(x_{0})+\E\|x_{0}\|_{2}^{2}<\infty.
\end{align*}
By Lemma \ref{lem:strong_approx-1} we then have
\begin{align*}
\E\vp(X_{t}^{\ve,n})+\E\int_{0}^{t}\|\partial_{x}\phi(\partial_{x}X_{r}^{\ve,n})\|_{H}^{2}dr & \le\E\vp(x_{0}^{n})+C,\\
 & \le\E\vp(x_{0})+C
\end{align*}
and we follow the same arguments as in Step 1 to pass to the limit.
\end{proof}
\appendix

\section{Relaxation of a linear growth functional with periodic boundary conditions\label{sec:periodic_relaxation}}

In this section we will prove that the functional
\[
\vp(v):=\begin{cases}
\int_{\mcO}\psi(Dv)dx+\frac{1}{2}\int_{\partial\mcO}|[v-v^{\bot}]|H^{d-1}(dx) & \text{if }v\in L^{2}\cap BV\\
+\infty & \text{if }v\in L^{2}\setminus BV,
\end{cases}
\]
where $\int_{\mcO}\psi(Dv)dx$ is defined as in Section \ref{sec:mc-pinned} is the lower-semicontinuous hull on $L^{2}$ of its restriction to $W_{per}^{1,1}(\mcO)$, where $\mcO=(0,1)$ (i.e. $d=1$). The arguments closely follow those from \cite[Fact 3.3]{A83} for the case of (inhomogeneous) Dirichlet boundary conditions. 
\begin{lem}
For all $u\in BV\cap L^{2}$ there exists a sequence of functions $u_{j}\in C^{1}\cap W_{per}^{1,1}\cap L^{2}$ such that 
\[
u_{j}\to u\quad\text{in }L^{2}
\]
and 
\[
\vp(u_{j})\to\vp(u).
\]
\end{lem}
\begin{proof}
Let $v_{j}\in C^{1}\cap W^{1,1}\cap L^{2}$ be a sequence of functions satisfying
\begin{align*}
v_{j} & \to u\quad\text{in }L^{2}\\
\int_{\mcO}|Dv_{j}|dx & \to\int_{\mcO}|Du|dx\quad\text{for }j\to\infty
\end{align*}
and $v_{j}=u$ on $\partial\mcO$ (cf. \cite[Theorem 10.1.2 and Remark 10.2.1]{ABM06}). We further define cut-off functions $w_{j}\in C^{1}\cap W^{1,1}\cap L^{2}$
\begin{align*}
w_{j}|_{\partial\mcO} & =\frac{u^{\bot}-u}{2}|_{\partial\mcO},\\
w_{j}(x) & =0,\quad\forall dist(x,\partial\mcO)>\frac{1}{j},\\
\int_{\mcO}|Dw_{j}|dx & \le\int_{\partial\mcO}|\frac{u^{\bot}-u}{2}|H^{d-1}(dx)+\frac{1}{j},\\
\int_{\mcO}|w_{j}|^{2}dx & \le\frac{C}{j}.
\end{align*}
Let $u_{j}=v_{j}+w_{j}$. Then $u_{j}(1)=u_{j}(0)=\frac{u(0)+u(1)}{2}$, in particular $u_{j}\in C^{1}\cap W_{per}^{1,1}\cap L^{2}$. Moreover,
\begin{align*}
u_{j} & \to u,\quad\text{in }L^{2}\\
\int_{\mcO}\sqrt{1+|Du_{j}|^{2}}dx & \to\int_{\mcO}\sqrt{1+|Du|^{2}}dx+\int_{\partial\mcO}|\frac{u^{\bot}-u}{2}|H^{d-1}(dx).
\end{align*}
We then complete the proof precisely as in \cite[Fact 3.3]{A83}.\end{proof}
\begin{lem}
For every $u\in L^{2}$ and every sequence $u_{j}\in BV\cap L^{2}$ with $u_{j}\to u$ in $L^{2}$ we have
\[
\liminf_{j}\vp(u_{j})\ge\vp(u).
\]
\end{lem}
\begin{proof}
The proof is the same as for \cite[Fact 3.4]{A83}.
\end{proof}

\appendix
\bibliographystyle{amsalpha.bst}
\bibliography{/home/benni/cloud/current_work/latex-refs/refs}

\def\cprime{$'$}
\providecommand{\bysame}{\leavevmode\hbox to3em{\hrulefill}\thinspace}
\providecommand{\MR}{\relax\ifhmode\unskip\space\fi MR }
\providecommand{\MRhref}[2]{%
  \href{http://www.ams.org/mathscinet-getitem?mr=#1}{#2}
}
\providecommand{\href}[2]{#2}
\begin{thebibliography}{ESvRS12}

\bibitem[ABM06]{ABM06}
Hedy Attouch, Giuseppe Buttazzo, and G{\'e}rard Michaille, \emph{Variational
  analysis in {S}obolev and {BV} spaces}, MPS/SIAM Series on Optimization,
  vol.~6, Society for Industrial and Applied Mathematics (SIAM), Philadelphia,
  PA, 2006, Applications to PDEs and optimization. \MR{2192832 (2006j:49001)}

\bibitem[ACM02]{ACM02}
Fuensanta Andreu, Vincent Caselles, and Jos{\'e}~Mar{\'{\i}}a Maz{\'o}n,
  \emph{A parabolic quasilinear problem for linear growth functionals}, Rev.
  Mat. Iberoamericana \textbf{18} (2002), no.~1, 135--185.

\bibitem[Anz83]{A83}
Gabriele Anzellotti, \emph{Pairings between measures and bounded functions and
  compensated compactness}, Ann. Mat. Pura Appl. (4) \textbf{135} (1983),
  293--318 (1984).

\bibitem[Anz85]{A85}
\bysame, \emph{The {E}uler equation for functionals with linear growth}, Trans.
  Amer. Math. Soc. \textbf{290} (1985), no.~2, 483--501.

\bibitem[BDPR09]{BDPR09}
Viorel Barbu, Giuseppe Da~Prato, and Michael R{\"o}ckner, \emph{Existence of
  strong solutions for stochastic porous media equation under general
  monotonicity conditions}, Ann. Probab. \textbf{37} (2009), no.~2, 428--452.

\bibitem[BR13]{BR13}
Viorel Barbu and Michael Röckner, \emph{Stochastic variational inequalities and
  applications to the total variation flow perturbed by linear multiplicative
  noise}, Arch. Ration. Mech. Anal. (2013), 1--38.

\bibitem[DLN01]{DLN01}
Nicolas Dirr, Stephan Luckhaus, and Matteo Novaga, \emph{A stochastic selection
  principle in case of fattening for curvature flow}, Calc. Var. Partial
  Differential Equations \textbf{13} (2001), no.~4, 405--425.

\bibitem[ESvR12]{ESR12}
Abdelhadi Es-Sarhir and Max-K. von Renesse, \emph{Ergodicity of stochastic
  curve shortening flow in the plane}, SIAM J. Math. Anal. \textbf{44} (2012),
  no.~1, 224--244.

\bibitem[ESvRS12]{ESRS12}
Abdelhadi Es-Sarhir, Max-K. von Renesse, and Wilhelm Stannat, \emph{Estimates
  for the ergodic measure and polynomial stability of plane stochastic curve
  shortening flow}, NoDEA Nonlinear Differential Equations Appl. \textbf{19}
  (2012), no.~6, 663--675.

\bibitem[FLP14]{FLP14}
Xiaobing Feng, Yukun Li, and Andreas Prohl, \emph{Finite element approximations
  of the stochastic mean curvature flow of planar curves of graphs}, Stochastic
  Partial Differential Equations: Analysis and Computations \textbf{2} (2014),
  no.~1, 54--83.

\bibitem[Ges12]{G12}
Benjamin Gess, \emph{Strong solutions for stochastic partial differential
  equations of gradient type}, J. Funct. Anal. \textbf{263} (2012), no.~8,
  2355--2383.

\bibitem[GR92]{GR92}
Donald Geman and George Reynolds, \emph{Constrained restoration and the
  recovery of discontinuities}, IEEE Transactions on pattern analysis and
  machine intelligence \textbf{14} (1992), no.~3, 367--383.

\bibitem[GT13]{GT11}
Benjamin Gess and Jonas~M. T{\"o}lle, \emph{Multi-valued, singular stochastic
  evolution inclusions}, arXiv:1112.5672, to appear in J. Math. Pures Appl.
  (2013).

\bibitem[KOJ05]{KOJ05}
Stefan Kindermann, Stanley Osher, and Peter~W. Jones, \emph{Deblurring and
  denoising of images by nonlocal functionals}, Multiscale Model. Simul.
  \textbf{4} (2005), no.~4, 1091--1115 (electronic).

\bibitem[LS98a]{LS98-2}
Pierre-Louis Lions and Panagiotis~E. Souganidis, \emph{Fully nonlinear
  stochastic partial differential equations}, C. R. Acad. Sci. Paris S\'er. I
  Math. \textbf{326} (1998), no.~9, 1085--1092.

\bibitem[LS98b]{LS98}
\bysame, \emph{Fully nonlinear stochastic partial differential equations:
  non-smooth equations and applications}, C. R. Acad. Sci. Paris S\'er. I Math.
  \textbf{327} (1998), no.~8, 735--741.

\bibitem[LS00]{LS00}
\bysame, \emph{Fully nonlinear stochastic {PDE} with semilinear stochastic
  dependence}, C. R. Acad. Sci. Paris S\'er. I Math. \textbf{331} (2000),
  no.~8, 617--624.

\bibitem[MR92]{MR92}
Zhi~Ming Ma and Michael R{\"o}ckner, \emph{Introduction to the theory of
  (nonsymmetric) {D}irichlet forms}, Universitext, Springer-Verlag, Berlin,
  1992.

\bibitem[PR07]{PR07}
Claudia Pr{\'e}v{\^o}t and Michael R{\"o}ckner, \emph{A concise course on
  stochastic partial differential equations}, Lecture Notes in Mathematics,
  vol. 1905, Springer, Berlin, 2007.

\bibitem[ROF92]{ROF92}
L.~Rudin, S.~Osher, and E.~Fatemi, \emph{Nonlinear total variation based noise
  removal algorithms}, Physica D: Nonlinear Phenomena \textbf{60} (1992),
  no.~1--4, 259--268.

\bibitem[RRW07]{RRW07}
Jiagang Ren, Michael R{\"o}ckner, and Feng-Yu Wang, \emph{Stochastic
  generalized porous media and fast diffusion equations}, J. Differential
  Equations \textbf{238} (2007), no.~1, 118--152.

\bibitem[SY04]{SY04}
P.~E. Souganidis and N.~K. Yip, \emph{Uniqueness of motion by mean curvature
  perturbed by stochastic noise}, Ann. Inst. H. Poincar\'e Anal. Non Lin\'eaire
  \textbf{21} (2004), no.~1, 1--23.

\end{thebibliography}

\end{document}